\documentclass[reqno]{amsart}
\usepackage{amssymb, amsmath}
\usepackage{natbib}
\usepackage{lscape,comment}
\usepackage{color}
\usepackage{dsfont}
\usepackage{mathrsfs}
\usepackage{bm}

\usepackage[pdftex,plainpages=false,colorlinks,hyperindex,bookmarksopen,linkcolor=red,citecolor=blue,urlcolor=blue]{hyperref}

\usepackage[left=3.5cm, right=3.5cm]{geometry}

\DeclareMathAlphabet{\mathpzc}{OT1}{pzc}{m}{it}

\bibpunct{[}{]}{;}{n}{,}{,}

\newtheorem{te}{Theorem}[section]

\newtheorem{os}[te]{Remark}
\newtheorem{prop}[te]{Proposition}
\newtheorem{lem}[te]{Lemma}
\newtheorem{ex}[te]{Example}

\numberwithin{equation}{section}

\def \l { \left( }
\def \r {\right) }
\def \ll { \left\lbrace }
\def \rr { \right\rbrace }

\def \mfB { \mathfrak{B} }

\newcommand{\bt}[1]{{\textcolor{black}{#1}}}

\newcommand{\rev}[1]{{\textcolor{black}{#1}}}

\newcommand{\revbt}[1]{{\textcolor{black}{#1}}}

\allowdisplaybreaks

\begin{document}

	\title[]{From semi-Markov random evolutions to scattering transport and superdiffusion}
	\author[]{Costantino Ricciuti}
	\address[]{Dipartimento di Scienze Statistiche, Sapienza - Universit\`{a} di Roma}
	\email[]{costantino.ricciuti@uniroma1.it}
	\author[]{Bruno Toaldo}
		\address[]{Dipartimento di Matematica ``Giuseppe Peano'' - Universit\`{a} degli Studi di Torino, Via Carlo Alberto 10 - 10123, Torino (Italy)}
	
	\email[]{bruno.toaldo@unito.it}
	\date{\today}

		\begin{abstract}
We here study random evolutions on Banach spaces, driven by a class of semi-Markov processes. The expectation (in the sense of Bochner) of such evolutions is shown to solve some abstract Cauchy problems. Further, the abstract telegraph (damped wave) equation is generalized to the case of semi-Markov perturbations.
A special attention is devoted to semi-Markov models of scattering transport processes which can be represented through these evolutions. In particular, we consider random flights with infinite mean flight times which turn out to be governed by a semi-Markov generalization of a linear Boltzmann equation; their scaling limit is proved to converge to superdiffusive transport processes.
\end{abstract}

	\maketitle

\tableofcontents

\section{Introduction}
In \cite{kac}, Mark Kac observed that if one considers a particle moving on a straight line at velocity $v$ and reversing direction at Poisson paced times, then this motion is governed by a pair of partial differential equations (PDEs). Most notably, he also realized that these two PDEs can be merged in the so-called telegraph (or damped waves) equation
\begin{align}
\partial^2_tq   +2\theta \partial_tq = v^2 \partial^2_xq.
\end{align}
The process above is called the telegraph process and, to the best of our knowledge, the first remarks in the modern mathematical literature appears in Goldstein \cite{goldstein}.

The $n$-dimensional version of such a process is the isotropic (Markovian) transport process (e.g., \cite{monin, Papanicolaou, Watanabe}). This is the uniform motion of a particle that chooses a new  direction with uniformly distributed angles, at any jump times of a Poisson process. The position-velocity density function solves a linear Boltzmann equation, (e.g., \cite{Watanabe}); moreover, by central limit arguments, a Brownian motion arises in the limit of large times and rapid jumps.

The above mentioned transport process can be seen in an abstract way: there is a running evolution (e.g., translation at velocity $v$) that changes mode of evolving (e.g., translation at velocity $v^\prime$) after exponentially distributed waiting times. This abstract idea led Griego and Hersh \cite{griegohersh} to formulate  the notion of Random Evolutions (see also \cite{Pinsky} and references therein).  Indeed one can imagine that there is a phenomenon whose instantaneous state is represented  by an element $u$ of a Banach space $\mfB$. The modes of time evolution are given by semigroups $\l T_v(t) \r_{t \geq 0}$, $v \in S$, on $\mfB$, and there is a random mechanism (e.g., a Markov chain $V(t)$ on $S$) which changes the mode of evolution  from $T_v(t)$ to $T_{v^\prime}(t)$ after an exponentially distributed waiting time. The authors realized the connection of these random evolutions with (systems of) abstract equations, for $v \in S$, ($S$ finite)
\begin{align*}
\partial_t q_v = G_v q_v + \theta_v\sum_{v^\prime \in S}  \l q_{v^\prime} - q_v\r h_{vv^\prime},  
\end{align*}
where $G_v$ generates $T_v(t)$ and $h_{vv^\prime}$ is the probability that a random jump of $V(t)$ starts from $v$ and arrives to $v^\prime$. The general formulation, for uncountable $S$, is
\begin{align}
\partial_t q_v = G_v q_v + \theta_v\int_{S}\l q_{v^\prime}-q_v \r h_v(dv^\prime),  
\label{eq primo ordine}
\end{align}
which reduces to a Boltzmann-type equation when $G_v = v \cdot \nabla_x$ on a suitable Banach space (see \cite[Corollary 3.1]{koro} for the general statement).
 Moreover, when $S= \ll v, -v \rr$,  they established that the PDEs above can be combined in the abstract telegraph equation
\begin{align}
\partial^2_tq  + 2\theta \partial_t q = v^2 G_v^2q, \label{eq secondo ordine}
\end{align}
where $q=1/2(q_v+q_{-v})$.

\bt{Exponential waiting times are typical in several physical systems, but this pattern can be distorted in many situations. From the probabilistic point of view this means that it is useful to relax the Markov assumption (with exponential time intervals)} in order to allow arbitrarily distributed waiting times between different
modes of evolution. Thus Korolyuk and Swishchuk had the idea of having $V(t)$ a semi-Markov process and they developed the theory of the so-called semi-Markov Random Evolutions  (see \cite{koro} and references therein). However, since in this case the Markov property is lost, the classical connection to abstract Cauchy problems is no longer true and a new theory in this sense has not yet been developed.

\rev{This paper makes a remarkable progress in this direction. Indeed, in Theorem \ref{teeqf} we find an abstract equation representing the semi-Markov counterpart of eq. \eqref{eq primo ordine} and therefore, it consists in a generalization of the abstract (linear) Boltzmann-type equation . We will also find the semi-Markov analogue of equation \eqref{eq secondo ordine} for the two-state model.}
It turns out that these equations are non-local in the time variable, as a consequence of the memory effect induced by the semi-Markov perturbation.

An application of the above theory gives us the possibility to develop a semi-Markov model of scattering transport: we consider a  semi-Markov version of the isotropic transport process,  i.e., whose flight times are not exponentially distributed. 
If such flight times have finite mean and variance, then this process is again an approximation of a Brownian motion, just like in the Markov case. Instead, the asymptotic behaviour in case of infinite mean and variance is more complicated and is not included in the limit theorems developed so far. 

Therefore we consider a random flight process whose flight times have infinite expectation and belong to the domain of attraction of a stable law.
First we show that this model of scattering transport is described by an integro-differential equation exhibiting a pseudo-differential operator in both space and time variables; such equation represents the  semi-Markov counterpart of the linear Boltzmann equation holding for the Markov flights.
We show that a suitable scaling of our transport process converges (in distribution) to a transport process with superdiffusive behaviour. At time $t$ this process is supported on the $d$-dimensional ball centered in the starting point and with radius $t$. Superdiffusive means that the mean square displacement of the limit process spreads, when $t \to \infty$, as $Kt^\gamma$, with $\gamma>1$ and $K>0$. In our case we will find that $\gamma=2$. This last result is obtained by adapting the limit theorems for coupled continuous time random walks developed in \cite{meer annals probab, meerstra}. It is noteworthy that the limit process is still a scattering transport process, performing, on any finite interval of time, a countable infinity of displacements shorter than $\epsilon>0$ and a finite number longer than $\epsilon$, for some $\epsilon>0$.
We stress that superdiffusion is empirically observed in many physical systems, like     turbulent diffusion, quantum optics, bacterial motions and many others (see \cite{Metzler} and references therein for an overview on this subject).

\section{Assumptions and preliminaries}
We briefly introduce here the foundation of the theory which will be used throughout the paper and we outline the basic assumptions under which our theory takes shape. We also here establish the notations used in the whole manuscript.
\subsection{Markov and Semi-Markov random evolutions} 
\label{sec21}
We refer to \cite{griegohersh, koro, Pinsky} for the basic theory.
Let $V(t)$ be a regular stepped semi-Markov process in the sense of \cite[Chapter 1]{koro}. Hence let $\l S, \mathfrak{S} \r$ be a metric space and let $ v_n$, $n \in \mathbb{N}$, be a discrete-time Markov chain on it which is embedded in $\l V(t), t \geq 0 \r$. The transition probabilities will be denoted as 
\begin{align}
\mathfrak{S} \ni E \mapsto h_v(E) := P^v \l v_1 \in E  \r    = P \l v_n \in E \mid v_{n-1} = v \r.
\end{align}
 Let $J_n$, $n \in \mathbb{N}$, be a sequence of non-negative r.v.'s with the distribution, for any $n \in \mathbb{N}$,
\begin{align}
F_v(w):= P^v \l J_1 \leq w \r:=P \l J_1 \leq w \mid v_1=v \r = P \l J_n \leq w \mid v_{n} = v \r.
\label{defF}
\end{align}
We will assume that \eqref{defF} are absolutely continuous, for any $v \in S$, and we will denote $g_v(w)$ a density. Further we will use the notation $\overline{F}_v(w):=1-F_v(w)$.
 Let $\tau_0:=0$ and $\tau_n:= \sum_{i=1}^n J_i$ for $n \in \mathbb{N}$. Hence denote
\begin{align}
V(t) = v_n \qquad   \tau_{n-1} \leq t  < \tau_{n}, \qquad n \in \mathbb{N}.
\label{defVntc}
\end{align}
The assumption of $V$ being regular means that it doesn't accumulate an infinity of jumps, i.e., if we define $N(t):= \max \ll n \in \mathbb{N} : \tau_n \leq t \rr$ we have $N(t)<\infty$, $P^v$-a.s. for any $v \in S$ and any $t>0$.

Now let $\mathcal{T}(t)$ be a semi-Markov random evolution in the sense of \cite{koro}. Hence, for each $v \in S$ let $\l T_v(t), t \geq 0 \r$, be a family of operators and assume that it forms a strongly continuous semigroup on a Banach space $\l \mfB, \left\| \cdot \right\| \r$. Now define the random operator on $\mfB$
\begin{align}
\mathcal{T}(t) \, := \, T_{V(t)}(t-\tau_{N(t)}) \cdots T_{v_2}(J_2)T_{v_1}(J_1).
\label{deft}
\end{align}
In the framework introduced in \cite{koro} the operator $\mathcal{T}(t)$ is called a `continuous' semi-Markov random evolution (see \cite[Def 3.2]{koro}).
Denote $\l G_v, \mathfrak{B}_0 \r$ the generators of $T_v(t)$ and suppose that $\mathfrak{B}_0  \subset \mfB$ is the common domain of definition of the operators $G_v$.
We remark that \eqref{deft} has the (stochastic) representation (see \cite[Lemma 3.1]{koro})
\begin{align}
\mathcal{T}(t) - \mathds{I} \, = \, \int_0^t G_{V(s)} \, \mathcal{T}(s) \, ds
\label{reprsemicont}
\end{align}
which must be meant on $\mathfrak{B}_0$, $\mathds{I}$ denoting the identity operator.

One of the most important objects in this paper is the mean value of a semi-Markov random evolution, i.e., for a function $u:S \mapsto \mfB$, the mapping
\begin{align}
 t \mapsto q_v(t) \, : = \, \mathds{E}^v \mathcal{T}(t)  u(V(t)) \in \mfB, \qquad t \geq 0,
\label{meanvalue}
\end{align}
where integration $\mathds{E}^v$ is meant in the Bochner sense.

If the $\l J_n, n \in \mathbb{N} \r$ are such that \eqref{defF} is the cdf of an exponential distribution with parameters $\theta_v$ then the process $V$ is a continuous time Markov chain and the operator $\mathcal{T}(t)$ in \eqref{deft} defines a Markov evolution in the sense of \cite{griegohersh, Pinsky}. In this case we will denote the process with $W(t)$ in place of $V(t)$; one can prove (e.g., \cite[Theorem 2]{griegohersh} for finite $S$ or \cite[Corollary 3.1]{koro} for general $S$) that $q_v(t)$ satisfies
\begin{align}
\partial_t q_v(t) \, = \, G_v q_v(t) + \theta_v \int_S \l q_{v^\prime}(t) - q_v(t) \r h_v(dv^\prime), \qquad q_v(0)=u(v).
\label{boltzmann}
\end{align}
Eq. \eqref{boltzmann} has the same form of a linear Boltzmann equation; indeed it reduces to a linear Boltzmann equation in case the evolution is given by translation semigroups on $\mathbb{R}^d$ at velocity $v$ ($G_v = v \cdot \nabla_x$), for an appropriate choice of $h_v$. We remark that $G_v$ and the integral operator on the right hand side of \eqref{boltzmann} act upon different variables and that the equation is meant on $\mfB_0$.

\subsection{Assumptions} 
\label{assumptions}
From now on we consider a special class of semi-Markov random evolutions, which is defined by the following assumptions.
\begin{itemize}
\item[A1)] For any $v\in S$, the family  $\l T_v(t), t \in \mathbb{R} \r$ forms a strongly continuous group of operators on $\l \mathfrak{B}, \left\| \cdot \right\| \r$,  such that  $\left\|T_v(t)u \right\|\leq \left\| u \right\|$ for all $t \in \mathbb{R}$ and $u \in \mfB$. 

\item[A2)] The semi-Markov process $V(t)$ is constructed as a continuous time Markov chain time-changed by the inverse of an independent driftless subordinator with infinite activity (i.e., a strictly increasing pure jump L\'evy process).
\end{itemize}
\textit{In Sections \ref{secre} and \ref{secdw} the assumptions A1) and A2) will be always considered fulfilled, without needing to specify it further.}
\subsubsection{Discussion}
We now discuss assumptions A1) and A2). Moreover, we introduce and discuss some further minor technical assumptions which will be sometimes requested (saying it expicitly).

First note that $A1)$ includes the remarkable case where $T_v$ represents a translation in  $\mathbb{R}^d$ at velocity $v$ (on suitable function spaces like $L^1 \l \mathbb{R}^d \r$ or $C_0 \l  \mathbb{R}^d\r$) and many others.
Since we work under the assumptions that the family $\l T_v(t) \r_{t \in \mathbb{R}}$ forms a group for any fixed $v$ we have that $-G_v$ generates $\l T_v(-t) \r_{t \geq 0}$ as well as $\l G_v \r$ generates $\l T_v(t)\r_{t \geq 0}$ in the sense of semigroups (see, for example, \cite[Section 3.11]{engelnagel}).

We now explain in detail the time-change construction for $V(t)$ to which we refer in A2).
Let $\sigma (t)$ be a subordinator, i.e., a one-dimensional L\'evy process with non-decreasing sample paths.  Its distribution is defined, for $\lambda >0$, by
\begin{align}
\mathds{E}e^{-\lambda \sigma (t)}= e^{-tf(\lambda)},
\end{align}
where $f(\lambda)$ is a Bernstein function (see more on subordinators in \cite{bertoins, librobern}). Hence $f(\lambda)$ has the form
\begin{align}
f(\lambda)= a+b\lambda+\int _0^\infty (1-e^{-\lambda s}) \nu (ds)
\label{reprbern}
\end{align}
where $a,b$ are non-negative constants and $\nu$ a sigma-finite measure fulfilling the integrability condition 
\begin{align*}
\int_0^\infty (s\wedge 1) \nu (ds) <\infty.
\end{align*}
The measure $\nu$ is said to be the L\'evy measure of $\sigma (t)$.
We here assume that $a=0$, which implies that $P(\sigma (t)<\infty)=1$ for all $t>0$ and that $b=0$, which implies that $\sigma (t)$ is a pure jump process with no drift. Since there is not drift, in order to require that the process $\sigma$ is strictly increasing we assume that $\nu(0, +\infty)=\infty$, i.e., the process has infinite activity.
Now, let 
\begin{align}
L(t)= \inf \{x>0: \sigma (x)>t\} \label{inverso}
\end{align}
be the inverse process of $\sigma (t)$. 
Moreover, let $\l W(t), t \geq 0 \r$ be a continuous time Markov chain on $S$, independent on \eqref{inverso}, which is completely defined by the embedded chain $( v_n, n\in \mathbb{N})$ and by waiting times having  exponential distribution with parameter $\theta_v$ such that 
\begin{align}
\sup_v \theta_v < \infty.
\label{sufin}
\end{align}
Condition \eqref{sufin} implies that $W(t)$ is non-explosive. In other words the process $W(t)$ is the Markov chain
\begin{align}
W(t) = v_n, \qquad \bar{\tau}_{n-1} \leq t < \bar{\tau}_{n}, \qquad n \in \mathbb{N}.
\end{align}
Hence, assumption A2) means that our semi-Markov process is defined by the following time-change
\begin{align}
\l V(t) \r_{t\geq 0}= \l W(L(t))\r_{t\geq 0}, \label{processo semi markov}
\end{align}
which means
\begin{align}
V(t) = W(s), \qquad \sigma(s-) \leq t < \sigma(s).
\label{tcq}
\end{align}
It is easy to see from \eqref{tcq} that the epochs (jump times) $\l \tau_n, n \in \mathbb{N} \r$ of $V(t)$ are a transformation of $\l \bar{\tau}_n,n \in \mathbb{N})\r$, i.e., $\tau_n = \sigma (\bar{\tau}_n-)$ a.s. However, by a simple conditioning argument, using independence and the fact that $\sigma(t)$ has no fixed discontinuity, i.e., $\sigma(t)-\sigma(t-)=0$, a.s., one has 
\begin{align}
\mathds{E}e^{-\lambda \sigma(\bar{\tau}_n-)} = \mathds{E} \mathds{E}\left[ e^{-\lambda \sigma(\bar{\tau}_n-)} \mid \bar{\tau}_n \right] = \mathds{E}e^{-\lambda \sigma(\bar{\tau}_n)}.
\end{align}
It follows that \eqref{defVntc} can be rewritten as
\begin{align}
V(t) = v_n \qquad \sigma(\bar{\tau}_{n-1}) \leq t < \sigma(\bar{\tau}_{n}), \qquad n \in \mathbb{N}.
\end{align}
Hence \eqref{processo semi markov} is characterized by the same embedded Markov chain $\{v_n\} _{n\in \mathbb{N}}$ of $W(t)$ but it  exhibits new waiting times $J_n$ such that 
\begin{align}
J_n= \tau_{n}-\tau_{n-1}= \sigma (\bar{\tau}_{n})-\sigma(\bar{\tau}_{n-1}), \qquad n \in \mathbb{N}. \label{intertempi}
\end{align}
By stationarity of increments of subordinators and since $\bar{\tau}_{n}-\bar{\tau}_{n-1}$ are exponentially distributed it is clear that
\begin{align}
F_v(w) =& P \l \sigma (\bar{\tau}_{n}-\bar{\tau}_{n-1}) \leq w \mid v_{n}=v\r \notag \\
= \, & 1- P \l \sigma(\bar{\tau}_{n}-\bar{\tau}_{n-1}) > w \mid v_{n} = v \r \notag \\
= \, & 1- P \l \bar{\tau}_{n}-\bar{\tau}_{n-1} > L(w) \mid v_{n}=v \r \notag \\
= \, & 1-\mathds{E} e^{-\theta_v L(w)}.
\label{217}
\end{align}
Since we assume that $F_v(w)$ has a Lebesgue density we will consider only subordinators whose one-dimensional marginal has a Lebesgue density. We will denote the density of $\sigma(t)$ with the symbol $\mu_t(w)$, i.e., 
\begin{align}
P \l \sigma(t) \in dw \r \, = \, \mu_t(w) dw.
\label{219}
\end{align}

A further quantity which is typical of semi-Markov processes is the \revbt{density defined by
\begin{align}
\mathpzc{h}_v(t)  \, dt \, := \, P^v \l \cup_{n } \ll \tau_n \in dt \rr \r,
\label{rendensgeneral}
\end{align}
which gives the probability that there is at least one jump during $dt$.} 
In what follows we will assume that
\begin{align}
t\mapsto \mathpzc{h}_v(t) \text{ is in } L^1_{\text{loc}} \l \mathbb{R}^+ \r,
\label{223}
\end{align}
for any $v$ \revbt{and (since we are working with non-explosive processes) also $\mathpzc{h}_v(t)= \sum_n \mathpzc{h}_v^n(t) := \sum_n P^v \l \tau_n \in dt \r/dt$}.
It is immediate to compute \eqref{rendensgeneral} whenever the process $N(t)$ is a renewal counting process, i.e., $\theta_v=\theta$, for some $\theta>0$. In this case the function $\mathpzc{h}_v(w)$ is the renewal density of $N(t)$ in the classical sense of renewal theory (e.g. \cite[page 26]{cox}) \revbt{and we have that}
\begin{align}
\mathpzc{h}_v(t)= \lim_{\Delta t \to 0}\frac{\mathds{E}^v \l N(t+\Delta t) - N(t) \r}{\Delta t}.
\label{rendensdef}
\end{align}
Since the process $N(t)$ is non-explosive the probability to have more than one jump in the interval $\Delta t$ is $o(\Delta t)$, and therefore the numerator of \eqref{rendensdef} can be interpreted as the probability of having one, or more, jump in the interval $\Delta t$. \revbt{In our case the computation can be conducted by exploiting the time-change construction, as follows.} 
By \eqref{219} we have that the renewal measure for our subordinators
\begin{align}
u^f(dw) := \mathds{E} \int_0^\infty \mathds{1}_{\left[ \sigma (t) \in dw \right]} dt
\label{potmeassub}
\end{align}
always has a density
\begin{align}
u^f(w):= \int_0^\infty \mu_s(w) ds.
\label{220}
\end{align}
This density is proportional to the renewal density $\mathpzc{h}_v(w)$ of the renewal counting process $N(t)$. Indeed,
\begin{align}
\mathpzc{h}_v(w)= \partial_w \mathds{E}N(w) \, = \, \partial_w \mathds{E}\mathcal{N}(L(w)) \, = \, \partial_w\theta \mathds{E}L(w)
\end{align}
where $\mathcal{N}$ is a Poisson process independent on $L(w)$. Since
\begin{align}
\mathds{E}L(w) \, = \, \mathds{E} \int_0^\infty \mathds{1}_{\ll \sigma (t) \leq w \rr}dt
\end{align}
it follows that $\mathpzc{h}_v(w)=\theta u^f(w)$ for any $v$. Further the density $u^f(w)$ is clearly in $L^1_{\text{loc}}\l \mathbb{R}^+ \r$.

\subsection{PDEs connection}
There is a wide literature considering processes of assumption A2) and their generalizations, in particular on the governing equations \cite{fracCauchy, baemstra, costaluisa, toniazzi, kololast, kochu,  magda, meerbounded, meertoa, orsrictoapota, orsrictoasemi, costafede, pierre, toaldopota, toaldodo}, their interplay with telegraph-type equations \cite{dovtel, orstel}, their interpretation as scale limit of CTRWs \cite{meer annals probab, KoloCTRW, meertri, meer phys, meerpoisson, meerstra, rictoa, strakavariable}, their distributional and path properties \cite{ascione, pierremladen, savtoa} and applications in different fields \cite{nicos, hairer, Metzler, scalas}.
In our scenario we can outline, heuristically, the PDEs connection as follows. Consider the fractional derivative $^MD_t^\alpha$ defined in the sense of Marchaud (also called generator definition \cite[eq. (2.15) page 30]{FCbook})
\begin{align}
^MD_t^\alpha u(t) \, = \, \frac{\alpha}{\Gamma (1-\alpha)} \int_0^\infty \l u(t)-u(t-s) \r s^{-\alpha -1} ds
\label{genform}
\end{align}
where $u(t-s)$ is meant to be zero as $t-s<0$. Note that in order to get from \eqref{genform} the canonical fractional derivative $D_t^\alpha$ one has to regularize \eqref{genform} as
\begin{align}
D_t^\alpha u(t) \, = \, \frac{\alpha}{\Gamma (1-\alpha)} \int_0^\infty \l u(t)-u(t-s) \r s^{-\alpha -1} ds -  \frac{s^{-\alpha}}{\Gamma (1-\alpha)} u(0).
\label{fracreg}
\end{align}
As the measure $\nu(ds)=\alpha s^{-\alpha-1}ds/\Gamma (1-\alpha)$ is the L\'evy measure associated with the Bernstein function $f(\lambda) = \lambda^\alpha$ and $s^{-\alpha}/\Gamma (1-\alpha) = \nu(s, \infty) =: \bar{\nu}(s)$, the operator \eqref{fracreg} can be easily generalized to an arbitrary Bernstein function $f$ with L\'evy measure $\nu(\cdot)$ as
\begin{align}
D_t^f u(t) \, = \, &  \, ^MD_t^f u(t) - \bar{\nu}(t) u(0)  \notag \\
= \, & \int_0^\infty \l u(t) - u(t-s) \r \nu(ds) - \bar{\nu} (t) u(0).
\label{226}
\end{align}
The mean value 
\begin{align*}
\pi(v,t)=\mathds{E}^v  u(V(t)), \qquad u\in C_0 (\mathbb{R}),
\end{align*}
is known to solve \rev{(see, in particular, \cite[Section 4]{KoloCTRW} or, for a slightly different approach, \cite[Section 5]{kochu}) the following equation}
\begin{align}
D_t^f \pi (v,t)  =  \theta_v\int _S( \pi (v',t)-\pi(v,t)) h _v(dv'), \qquad \pi(v,0)= u(v).\label{eq V(t)}
\end{align}
Clearly (\ref{eq V(t)}) is a generalization of  the Kolmogorov backward equation holding in the Markov case
\begin{align*}
\partial_t \pi (v,t)  = \theta_v \int _S( \pi (v',t)-\pi(v,t)) h _v(dv'), \qquad \pi(v,0)= u(v),
\end{align*}
whose unique solution is $\pi (v,t)= \mathds{E}^v u(W(t))$.

\section{Boltzmann-type equations}
\label{secre}
In this section we derive rigorously the governing equation for the expectation function $q_v(t)$ defined in \eqref{meanvalue}. Hence we should make rigorous the following assertion: the function $q_v(t)$ satisfies, on $\mfB$,
\begin{align}
f \l \partial_t-G_v  \r \l q_v(t) - T_v(t) q_v(0) \r\, = \, \theta_v\int_S \l q_{v^\prime} (t)- q_v(t) \r h_v(dv^\prime), 
\label{oureq}
\end{align}
subject to $q_v(0) = u(v)$, where $f$ is the Laplace exponent defined in \eqref{reprbern}. 
Whenever $f(\lambda) = \lambda$ (hence there is no time-change and $V(t)$ is Markovian) one recovers, formally, the linear Boltzmann-type equation
\begin{align}
\partial_t q_v(t) - G_v q_v(t)\, = \,  \theta_v\int_S \l q_{v^\prime} (t)- q_v(t) \r h_v(dv^\prime).
\end{align}
In this section we proceed as follows. In Section \ref{secdeff} we address the problem of defining the operator $f \l \partial_t - G_v  \r$ appearing in the left hand side of \eqref{oureq}; as $f$ is a Bernstein function, we take inspiration from the theory of Bochner subordination, and the corresponding functional calculus, whose basic facts will be outlined at the beginning of Section \ref{secdeff} (the reader can consult \cite[Chapter 12]{librobern} or \cite[Chapter 2]{jacob} for a thorough discussion on Bochner subordination). Then, in Section \ref{secexp}, we obtain some technical properties of $q_v(t)$ which will be needed throughout the paper.

\subsection{The operator $f \l \partial_t - G_v  \r$ through Bochner subordination}
\label{secdeff}
Take a family of operators $\l T_t \r_{t \geq 0}$ and suppose that it forms a strongly continuous contraction semigroup on $\mfB$ and define
\begin{align}
Au:=\lim_{t \to 0+} \frac{T_tu-u}{t} 
\label{defgen}
\end{align}
the generator of $T_t$ on the domain Dom$(A)$, i.e.,
\begin{align}
\text{Dom}(A):=\ll u \in \mfB : \text{ the limit  \eqref{defgen} exists as strong limit} \rr.
\end{align}
Then take a subordinator $\sigma(t)$, $t\geq 0$, with Laplace exponent $f$, and define the family $( T_t^f )_{t \geq 0}$ as
\begin{align}
T_t^f u : = \int_0^\infty T_s u \, P \l \sigma(t) \in ds \r
\label{bochsub}
\end{align}
where the integral \eqref{bochsub} is meant as a Bochner integral. The Phillips' theory states that the family $T_t^f$ is still a semigroup (this comes from Markov property of $\sigma (t)$) and that the generator of $T_t^f$ is given by the operator $ A^f $ (with domain Dom$(A^f)$) such that
\begin{align}
A^fu \big|_{\text{Dom}(A)} \, = \, -f (-A) u \, := \, \int_0^\infty \l T_s u -u \r \nu(ds)
\label{phil}
\end{align}
and, in general, $\text{Dom}(A^f) \supset \text{Dom}(A)$ (for all the assertions above, see \cite[Propositions 12.1, 12.5 and Theorems 12.6]{librobern}). \rev{We observe that formula \eqref{phil} uses the representation \eqref{reprbern} for the Bernstein function $f$ (with $a=b=0$) and the basic definition of pseudo-differential operators.}

We note that, on functions $L^1 \l [0,T]; \mathbb{R} \r$, for any $T>0$, the operator $^MD^f_t$ defined in \eqref{226} can be interpreted by means of the above theory.
Indeed, if we define the family of operators $\l \Gamma_t \r_{t \geq 0}$ as
\begin{align}
\Gamma_s u (t) \, = \, \begin{cases} u(t-s), \qquad &s \leq t, \\ 0, & s > t,
\end{cases}
\label{defgamma}
\end{align}
then it is well known that this family forms a strongly continuous contraction semigroup on $L^1\l [0,T]; \mathbb{R} \r$, for $T>0$. Hence one might be tempted to write, in the spirit of Bochner subordination, $^MD_t^f u (t)=f \l -A^\Gamma \r$, where $A^\Gamma$ denotes the generator of $\Gamma_t$. However the generator of the (killed) translation $\Gamma_s$ is defined on functions that are differentiable, in appropriate sense, and such that $u(0)=0$. So we apply Phillips' formula as in \eqref{phil} to the function $u^0(t):=u(t)-u(0)$ and this yields the representation
\begin{align}
D_t^f u (t) \,  = \,& -\int_0^\infty \l  u(t-s) \mathds{1}_{[s \leq t]} -u(t)\r \, \nu (ds) - \bar{\nu}(t) u(0) \notag \\
= \, &- \int_0^\infty \l \Gamma_s u^0(t) - u^0(t) \r \nu(ds) \notag \\
= \, & f \l -A^\Gamma \r \l u(t)-u(0) \r
\label{spieg}
\end{align}
and the integral makes sense as a Bochner integral because of \eqref{phil}, under suitable assumptions on $u(t)$.

In this paper we consider a generalization of \eqref{spieg} on $L^1 \l [0,T]; \mfB \r$, $T>0$, i.e., the space of functions $[0,T] \ni t \mapsto u(t) \in \mfB$ with finite $L^1$ norm $\int_0^T \left\| u(t) \right\|dt$. \rev{The reader can consult [28, Section 5.13] for a general theory of semigroups acting on Banach-valued functions (see in particular Proposition 5.13.1 and Theorem 5.13.1 of which the forthcoming Lemmas 3.1 and 3.3 are analogs for translations on $L^1 \l [0,T], \mathfrak{B} \r$).}

Indeed consider the operators
\begin{align}
\Gamma_s u (t) \, = \, \begin{cases} u(t-s), \qquad &s \leq t, \\ \bm{0}, & s > t,
\end{cases}
\label{defgammagen}
\end{align}
on $L^1 \l [0,T]; \mfB \r$, $T>0$.
It turns out that this family is a strongly continuous contraction semigroup, as it is outlined in the forthcoming result; in the following the derivative is meant on absolutely continuous functions $u:[0, T] \mapsto \mathfrak{B}$ with the representation
\begin{align}
u(t)-u(0) = \int_0^t u^\prime (s) ds
\label{reprac}
\end{align}
for any $t \in [0,T]$.
We remind that in this framework a function $u$ is said to be absolutely continuous if for every $\epsilon >0$ there exists $\delta>0$ such that for every finite collection $\ll (a_i, b_i) \rr$ of disjoint intervals in $[0,T]$ with $\sum_i(b_i-a_i)< \delta$ it is true that $\sum_i \left\| u(b_i)-u(a_i) \right\| < \epsilon$. We denote this class of functions as $AC([0,T];\mathfrak{B})$. We remark that the fundamental theorem of calculus as in the scalar-valued case does not hold for Bochner integral, which has instead a weaker version of it (see \cite[Proposition 1.2.3]{abhn}.) Hence a function $u \in AC\l [0,T]; \mathfrak{B} \r$ is not necessarily a.e. differentiable and the representation \eqref{reprac} is not necessarily true.

In view of the previous heuristic discussion it is clear that we will need the following technical Lemma.
\begin{lem}
\label{lemmatrasl}
The family of operators $\Gamma_t$, $t \geq 0$, defined in \eqref{defgammagen} forms a strongly continuous contraction semigroup on $L^1 \l [0,T]; \mathfrak{B} \r$, for any $T>0$. The generator $\l A^\Gamma, \text{Dom} (A^\Gamma) \r$ is given by $A^\Gamma u = - u^\prime$ with $\text{Dom} (A^\Gamma) = W$ where
\begin{align}
W := &\ll u \in L^1 \l [0,T], \mathfrak{B}  \r : u(t)-u(0) = \int_0^t u^\prime(s) ds\,  \text{ for all } t \in [0,T]  \text{ and } u(0) = \bm{0}  \rr.
\end{align}
\end{lem}
\begin{os} \normalfont
We remark that $W$ coincides with $AC\l [0,T], \mfB \r$ (with $u(0)=\bm{0}$), whenever the Banach space $\mathfrak{B}$ has the Radon-Nikodym property, i.e., when the Banach space is such that absolutely continuous functions are a.e. differentiable (see \cite[Definition 1.2.5]{abhn}). Otherwise $W$ is a subset of it.
\end{os}
\begin{proof}[Proof of Lemma \ref{lemmatrasl}]
Fix $T>0$ arbitrarily.
First we prove that $\Gamma_t$ is a contraction. We have that
\begin{align}
\int_0^T \left\| \Gamma_s u (t) \right\| \, dt \, = \,& \int_0^T \left\| u(t-s) \right\| \mathds{1}_{[s \leq t]} \, dt \notag \\
= \, & \mathds{1}_{[s \leq T]} \int_0^{T-s} \left\| u(t) \right\| dt \notag \\
\leq \, & \int_0^T \left\| u(t) \right\| dt.
\end{align}
Strong continuity of $\Gamma_s$ on $L^1 \l [0,T], \mathfrak{B}  \r$ can be proved by checking it first on the set of continuous functions with compact support which is dense in $L^1 \l [0, T], \mathfrak{B} \r$ and then extending to the whole Banach space (see, for example, \cite[Proposition 5.3]{engelnagel} or the discussion in \cite[page 14]{abhn}). To obtain the generator and its domain we proceed as follows. Let $\l \widehat{A}, \text{Dom}(\widehat{A})\r$ be the operator $\widehat{A}u = -u^\prime$ on $W$, we first check that Dom$(A^\Gamma) \subset W$ and that $\widehat{A}|_{\text{Dom}(A^\Gamma)}=A^\Gamma$. Note that, for any $t\in [0,T]$ and $u \in \text{Dom}(A^\Gamma)$ one has
\begin{align}
- h^{-1}  \int_{t-h}^{t} u(s) ds \, = \, h^{-1} \int_{0}^{t} (u(s-h)\mathds{1}_{[s-h \geq 0]} - u(s) ) ds.
\end{align} 
As $h \to 0$ the rhs converges to $\int_0^t g(s) ds$, where $g(s):=A^\Gamma u(s)$, since integration over compact intervals is continuous on $L^1$ while the lhs instead converges to $-u(t)$ for almost all $t\in [0, T]$. If we define appropriately $u$ on a null set we get
\begin{align}
u(t) = \int_0^t (-g(s)) ds
\end{align}
and thus $u$ is an $L^1 \l [0,T], \mathfrak{B} \r$ function and is absolutely continuous (according to \cite[Proposition 1.2.2]{abhn}) with derivative $-A^\Gamma u(s)$ and $u(0) = 0$. This shows that Dom$(A^\Gamma) \subset W$ and that $\widehat{A}|_{\text{Dom}(A^\Gamma)}=A^\Gamma$.
Denote now $\rho(\widehat A)$ and $\rho(A^\Gamma)$ the resolvent sets of the two operators. It is easy to see that the resolvent operator of $\widehat{A}$ is,
\begin{align}
R(\lambda, \widehat{A}) u(t)\, = \, \int_0^t e^{-\lambda (t-s)} u(s)ds, \qquad t \in [0,T],
\end{align}
for $\lambda >0$. It follows easily that $\rho(\widehat A) \cap \rho(A^\Gamma) \neq \emptyset$ and thus from \cite[Exercise IV.1.21.(5)]{engelnagel} we can conclude that $\l A^\Gamma, \text{Dom}(A^\Gamma) \r = \l \widehat{A}, W \r$.
\end{proof}
We are now in position to define the operator $f \l \partial_t-G_v \r$. 

\emph{As $-A^\Gamma$ is the differentiation operator on absolutely continuous functions, in the rest of the paper we will use the notation $-\partial_t$ in place of $A^\Gamma$.}

For any $v \in S$ we consider the family $\l U_s^v \r_{s \geq 0}$ of operators on functions $u \in L^1 \l [0,T]; \mathfrak{B} \r$ given by
\begin{align}
U_s^v u (t) \, := \, T_v(s) \Gamma_s u(t)
\end{align}
As a consequence of Lemma \ref{lemmatrasl} it is easy to prove that the operators $U_s^v$ form a strongly continuous contraction semigroup on $L^1 \l [0,T]; \mathfrak{B} \r$ whose generator has the form $A^U_vu(t)= -\partial_t u(t) + G_vu(t)$ on an appropriate subset of $ \text{Dom}(A^U_v)$, and thus we will use Phillips formula to define
\begin{align}
 f\l -(-\partial_t + G_v) \r u (t)\, := \,&f (-A^U_v)u (t) \notag \\ := \,  & -\int_0^\infty \left[U_s^vu(t)-u(t) \right] \nu(ds),
\label{defop}
\end{align}
for suitable functions $u$ (and thus with $u(0)=\bm{0}$).
\begin{lem}
\label{lemmadt}
Fix $T>0$ arbitrarily. The family $\l U_s^v \r_{s \geq 0}$ forms a strongly continuous contraction semigroup on $L^1 \l [0,T]; \mathfrak{B} \r$. The generator $\l A^U_v, \text{Dom}(A^U_v) \r$ is such that
\begin{align}
H^v \, := \, \ll u \in L^1 \l [0,T]; \mathfrak{B}\r :  u \in W, u(t) \in \mfB_0 \text{ for any } t \geq 0 \text{ and } G_vu  \in L^1 \l [0,T]; \mathfrak{B}\r  \rr
\label{defhv}
\end{align}
is a subset of $\text{Dom}(A^U_v)$ and $A^U_v \big|_{H^v} = -\partial_t + G_v$.
\end{lem}
\begin{proof}
Properties follow from Lemma \ref{lemmatrasl}. Indeed
\begin{align}
\int_0^T \left\| U_s^v u(t) \right\| dt\, = \,& \int_0^T \left\| T_v(s)\Gamma_su(t) \right\| dt \notag \\
\leq \, & \int_0^T\left\| \Gamma_su(t) \right\| dt \notag\\
\leq \, & \int_0^T\left\| u(t) \right\| dt
\end{align}
where the last inequality follows from Lemma \ref{lemmatrasl} while the second last from the fact that $T_v(s)$ is a contraction on $\mathfrak{B}$.
The property $U_s^vU_t^vu = U_{s+t}^vu$ is easily checked. Strong continuity instead can be checked by observing that
\begin{align}
&\int_0^T \left\| U_s^vu(t)-u(t) \right\| dt \notag \\ = \,& \int_0^T \left\| T_v(s)\Gamma_su(t)-u(t) \right\| dt \notag \\
\leq \, & \int_0^T \left\| T_v(s)\Gamma_su(t)-T_v(s)u(t)\right\| dt+ \int_0^T \left\| T_v(s)u(t)-u(t) \right\| dt \notag \\
\leq \, & \int_0^T \left\| \Gamma_s u(t) -u(t) \right\| dt + \int_0^T \left\| T_v(s)u(t)-u(t) \right\| dt
\label{verophil}
\end{align}
and then sending $s \to 0$ in \eqref{verophil}: the first member goes to zero by Lemma \ref{lemmatrasl} while the second goes to zero by dominated convergence (e.g., \cite[Theorem 1.1.8]{abhn}) since $s \mapsto T_v(s)u(t)$ is continuous as a function $[0,\infty] \ni s \mapsto \mathfrak{B}$ and $\left\| T_v(s)u-u \right\| \leq 2 \left\| u(t) \right\|$. 

Supposing now that $u \in H^v$, we have that
\begin{align}
\lim_{h \downarrow 0} \frac{U_h^vu-u}{h}
\end{align}
exists as strong limit in $L^1 \l [0,T]; \mathfrak{B} \r$.
Note that, indeed,
\begin{align}
\lim_{h \downarrow 0} \frac{U_h^v u-u}{h} \, = \, &\lim_{h \downarrow 0} \frac{T_v(h) \Gamma_h u - T_v(h)u+T_v(h)u-u}{h} \notag \\
= \, & \lim_{h \downarrow 0} T_v(h)\frac{ \Gamma_h  - I}{h} u+\frac{T_v(h)u-u}{h}  \notag \\
= \, & \lim_{h \downarrow 0} \left[ T_v(h) A^\Gamma_h u  + \frac{T_v(h)u-u}{h} \right],
\label{eddau}
\end{align}
where $A_h^\Gamma u := \l \Gamma_h - I \r/h$.
The last member in \eqref{eddau} is $G_vu$ as $h \to 0$ by dominated convergence since, whenever $u(t) \in \mfB_0$, one has $h^{-1}\left\| T_v(h)u(t)-u(t) \right\| \leq \left\| G_vu \right\|$. For the first note that
\begin{align}
&\int_0^T \left\| T_v(h) A_h^\Gamma u(t) - \l -\partial_t \r u(t) \right\| dt \notag \\
= \,& \int_0^T \left\| T_v(h) A_h^\Gamma u(t) -T_v(h) \l -\partial_t \r u(t)  +T_v(h) \l -\partial_t \r u(t) - \l -\partial_t \r u(t)\right\| dt \notag \\
\leq \, & \int_0^T \left\| A_h^\Gamma u(t) - \l -\partial_t \r  u(t)\right\| dt  + \int_0^T \left\| T_v(h)\l -\partial_t \r u(t) - \l -\partial_t \r u(t) \right\| dt
\end{align}
goes to zero as $h \to 0$:  the first member goes to zero since $u(t)\in \text{Dom}(-\partial_t)$ while the second goes to zero by dominated convergence for Bochner integrals since $T_v(h)$ is strongly continuous (on $\mathfrak{B}$).
\end{proof}
With the properties of $U_s^v$ at hand, we are finally able to define $f \l \partial_t -G_v \r$ as
\begin{align}
f \l \partial_t -G_v \r  \,: = \, f \l -A^U_v \r \big|_{H^v} \, = \, -\int_0^\infty \l U_s^v - \mathds{I} \r \, \nu(ds),
\end{align}
since on $H^v$ the generator $A^U_v$ reduces to $-\partial_t+G_v$.
In the following Lemma we clarify this in terms of Bochner subordination.
\begin{lem}
\label{domfg}
The subordinate semigroup
\begin{align}
\mathfrak{U}^v_t \, := \, \int_0^\infty U^v_s P \l \sigma (t) \in ds \r
\end{align}
is generated by $\mathfrak{A}_v$ with domain $\text{Dom}(\mathfrak{A}_v) \supset H^v $, such that
\begin{align}
\mathfrak{A}_v \big|_{H^v} \, = \,- f \l \partial_t -G_v \r.
\end{align}
\end{lem}
\begin{proof}
Phillips' theorem implies that $\text{Dom}(A_v^U) \subset\text{Dom}(\mathfrak{A}_v)$ and
\begin{align}
\mathfrak{A}_v \big|_{\text{Dom}(A_v^U)} = \, & -f \l -A_v^U \r \notag \\
= \, & \int_0^\infty \l U_s^v -\mathds{1} \r \nu(ds).
\end{align}
Since $H^v \subset \text{Dom}\l A_v^U \r$ and, by Lemma \eqref{lemmadt}, $A_v^U \big|_{H^v} = -\partial_t +G_v$ then
\begin{align}
\mathfrak{A}_v \big|_{H^v} \, = \, - f \l \partial_t -G_v \r.
\end{align}
\end{proof}

\subsection{The expectation of a semi-Markov random evolution}
\label{secexp}
We obtain here some properties of
\begin{align}
q_v(t) \, : = \, \mathds{E}^v \mathcal{T}(t) u(V(t))
\end{align}
for suitable functions $u:S \mapsto \mathfrak{B}$.

\begin{lem}
\label{qindomg}
Let $u:S \mapsto \mathfrak{B}_0$ be such that
\begin{align}
\sup_{\substack{  v_1 \cdots v_{n} \\ t_1 \cdots t_{n} \\ \rev{n \in \mathbb{N}}}} \left\| G_{v_1} T_{v_{n}}(t_{n}) \cdots T_{v_1}(t_1) u(v_n)  \right\| < \infty.
\label{assbrutta}
\end{align}
Then we have that $q_v(t) \in \mfB_0$ for any $v \in S$ and $t \geq 0 $ and further $G_vq_{v} \in L^1 \l [0,T]; \mathfrak{B} \r $ for any $v \in S$.
\end{lem}
\begin{proof}[Proof of Lemma \ref{qindomg}]
We have by the assumptions that $\mathcal{T}(t,\omega) u(V(t,\omega)) \in \mathfrak{B}_0$ and further that $\left\| G_v\mathcal{T}(t) u(V(t)) \right\| < C$ where $C$ does not depend on $\omega$ and $t$. It follows that $G_v \mathcal{T}(t)u(V(t))$ is $P^v$-Bochner integrable\rev{. T}hen $G_v \mathds{E}^v \mathcal{T}(t)u(V(t)) = \mathds{E}^v G_v\mathcal{T}(t)u(V(t))$ since $G_v$ is linear and closed. It follows that  $\mathds{E}^v \mathcal{T}(t)u(V(t)) \in \mfB_0$ and also
\begin{align}
\int_0^T \left\| G_vq_v(t) \right\| dt \, \leq \, CT.
\label{331}
\end{align}
\end{proof}
Assumption \eqref{assbrutta} is satisfied in many situations of interest, such as the cases of translation and rotation groups, as shown in the following examples.
\begin{ex}[Translation] \normalfont
\label{extrans}
Let $S$ be finite and let $T_v(t)$ be a translation on $\mfB=L^1(\mathbb{R})$ at velocity $v$, say $T_v(t)h(x)= h(x+vt)$ (see, e.g.,  the case of the random evolution driven by the ``telegraph'' process, treated in  section 4, where $S$ contains two elements). Use the notation $u(v_n)=h_{v_n}$. Then, for $h_{v_n} \in AC (\mathbb{R})$, such that $h'_{v_n}\in L^1(\mathbb{R})$ we have
\begin{align}
\left\| G_{v_1} T_{v_n}(t_n)\dots T_{v_1} (t_1)h _{v_n} \right\|_{L^1} =    \int _{\mathbb{R}} \bigl |v_1 \partial_x h_{v_n}  (x+v_1t_1+\dots +v_nt_n) \bigr |dx =|v_1|\, \left\|h'_{v_n} \right\|_ {L^1}   
\end{align}
as the $L^1$ norm is translation-invariant. Hence we have
\begin{align}
\sup_{\substack{ v_1 \cdots v_{n} \\ t_1 \cdots t_{n} \\ \revbt{n \in \mathbb{N}} }} \left\| G_{v_1} T_{v_n}(t_n) \dots T_{v_1} (t_1) h _{v_n} \right\| _{L^1} = \sup _{v_1, v_n }|v_1| \,  \left\| h'_{v_n} \right\|_ {L^1}<\infty
\end{align}
the last inequality holding since $S$ is finite. A similar argument applies to the case where $T_v(t)$ is a translation on $\mfB=C_0(\mathbb{R})$, endowed with the sup-norm $ \left\|\cdot \right\|_{\infty}$. Indeed, for $h_{v_n}\in C_0^1(\mathbb{R})$, we have
\begin{align}
\left\|G_{v_1} T_{v_n}(t_n)\dots T_{v_1}(t_1)h _{v_n} \right\| _{\infty} =  \sup _x  |v_1 \frac{\partial}{\partial x} h_{v_n}  (x+v_1t_1+\dots +v_nt_n)| =|v_1|\,  \left\|h'_{v_n}  \right\|_ {\infty}   
\end{align}
since the sup-norm is invariant under translations. Then, by finiteness of $S$, we have
\begin{align}
\sup_{\substack{ v_1 \cdots v_{n} \\ t_1 \cdots t_{n} \\ \revbt{n \in \mathbb{N}}}} \left\| G_{v_1} T_{v_n}(t_n)\dots T_{v_1}(t_1)h _{v_n} \right\| _{\infty} = \sup _{v_1, v_n }|v_1| \,  \left\| h'_{v_n} \right\|_ {\infty}<\infty.
\end{align}
\end{ex}
\begin{ex}[Rotation] \normalfont
Let  $T_v(t)$ be  a rotation in $L^1 \l \mathbb{R}^2 \r$ (or $C_0 \l \mathbb{R}^2 \r$) defined by 
\begin{align}
T_v(t) h(x_1, x_2)= h(x_1\cos t+v x_2\sin t\, ,  \,  -vx_1\sin t+x_2\cos t), \qquad t\geq 0, \qquad v=\pm 1, \label{rotazione}
\end{align}
where $v=\pm 1$ respectively denote the clockwise and the counterclockwise rotation, with generator
\begin{align*}
G_v h(x_1, x_2)= v \l x_2\partial_{x_1} h(x_1, x_2)-x_1 \partial_{x_2} h(x_1, x_2) \r.
\end{align*}
Similarly to Example \ref{extrans}, it is easy to prove that  assumption \eqref{assbrutta} is satisfied in this case of rotation operators, because the $L^1$ and the sup-norm are invariant even under rotations; we leave the proof to the interested reader.
\end{ex}

The following result characterizes continuity and differentiability properties of $q_v(t)$.
\begin{lem}
\label{lemmadiff} 
Assume that $u: S \mapsto \mathfrak{B}_0$ is such that $\sup_v \left\| u(v) \right\| <   \infty$ and that the assumptions in Lemma \ref{qindomg} are satisfied. Suppose further that $\mathpzc{h}_v(w)$ satisfies \eqref{223}. Then, for any $v$ and $T>0$, it is true that, $q_v(t) \in AC\l [0,T]; \mathfrak{B} \r$, $q_v^\prime$ exists a.e. and is Bochner integrable on $[0,T]$ and further
\begin{align}
q_v(t) - q_v(0) \, = \, \int_0^t q_v^\prime (s) ds, \qquad t \in [0,T].
\end{align} 
\end{lem}
\begin{proof}
Denote $\gamma (t) : = t-\tau_{N(t)}$ and $Bu(v):=\int_S \l u(v^\prime)-u(v) \r h_v(dv^\prime)$ and use the Dynkin-type representation of the semi-Markov random evolution $\mathcal{T}(t)$ which can be found (for example) in \cite[Corollary 2.5]{suisciucbook}: for $s>0$,
\begin{align}
q_v(s) - q_v(0) \, = \, &\mathds{E}^v\int_0^s   \l \frac{g_{V(y)}(\gamma(y))}{\overline{F}_{V(y)}(\gamma(y))} \mathcal{T}(y) Bu(V(y)) + \mathcal{T}(y) G_{v} u(V(y))\r \, dy \notag \\
= \, & \int_0^s  \mathds{E}^v \l \frac{g_{V(y)}(\gamma(y))}{\overline{F}_{V(y)}(\gamma(y))} \mathcal{T}(y) Bu(V(y)) + \mathcal{T}(y) G_{v} u(V(y))\r \, dy,
\label{516}
\end{align}
where $g$ and $\overline{F}$ are the densities and the survival functions of the waiting times according to Section \ref{sec21}.
The last equality is justified by the following arguments.
Since $T_v(t)u(v^\prime) \in \mathfrak{B}_0$ for any $v$, $v^\prime$ and $t\geq 0$, by our assumptions, we can find constants, say it $C_1$ and $C_2$, such that
\begin{align} 
&\left\| \mathcal{T}(y) Bu(V(y)) \right\| \, \leq \, \left\| Bu(V(y)) \right\| \, \leq 2 \sup_v\left\| u(v) \right\| \, \leq \, C_1 < \infty \label{c1} \\
& \left\| \mathcal{T}(y) G_v u(V(y)) \right\| \, \leq \, \left\| G_vu(V(y)) \right\| \, \leq \, C_2 < \infty \label{c2}
\end{align}
and further $C_1$ and $C_2$ depend on $u$ but are independent from $y$.
Therefore, using that $\overline{F}_v$ is non increasing and that a.s. $\gamma(y)\in [0,y]$,
\begin{align}
& \int_0^s \left\| \mathds{E}^v \l \frac{g_{V(y)}(\gamma(y))}{\overline{F}_{V(y)}(\gamma(y))} \mathcal{T}(y) Bu(V(y)) + \mathcal{T}(y) G_vu(V(y))\r\right\| dy \notag \\
\leq \,&  \frac{C_1}{\overline{F}_\star(s)} \int_0^s \mathds{E}^v g_{V(y)}(\gamma(y)) dy  + s C_2.
\label{herewe}
\end{align}
where, using \eqref{217}, it is easy to see that $\overline{F}_\star(s):= \inf_v \overline{F}_v(s)>0$ because $\sup_v \theta_v < \infty$.
Since the waiting times between jumps are absolutely continuous r.v.'s, then the distribution of $\gamma (y)$ can be estimated as in \cite[page 61]{cox} to be 
\begin{align}
P^v \l \gamma (y) \in dw \r \leq \delta_y(dw) +   \mathpzc{h}_v(y-w) dw.
\end{align}
\revbt{It follows that
\begin{align}
\mathds{E}^vg_{V(y)}(\gamma(y)) \, = \,& \int_S\int_0^y g_z(w) P^v \l \gamma(y) \in dw, V(y) \in dz \r \notag \\
= \, &\int_S \int_0^y \int_0^{+\infty} \mu_s(w) \theta_ze^{-\theta_zs}ds P^v \l \gamma(y) \in dw, V(y) \in dz \r \notag \\
\leq \, &  \sup_z \theta_z \int_0^y u^f (w)P^v \l \gamma (y) \in dw \r \notag \\
\leq \, & \sup_z \theta_z \int_0^y u^f(w) \l \delta_y(dw) + \mathpzc{h}_v(y-w)  dw \r \notag \\
= \,& \sup_z\theta_z \l  u^f(y) + \int_0^y u^f (w)  \mathpzc{h}_v(y-w) dw\r
\label{stimamediagg}
\end{align}
}where $u^f$ is defined in \eqref{220}. Since $u^f \in L^1_{\text{loc}}\l \mathbb{R}^+ \r$ it follows that the second term  in \eqref{stimamediagg} is in $L^1 \l \mathbb{R}^+ \r$ by the properties of Laplace convolution (e.g. \cite[page 22]{abhn}). Therefore $y \mapsto \mathds{E}^v \l g_{V(y)} (\gamma(y)) \r$ is in $L^1_{\text{loc}} \l \mathbb{R}^+ \r$ for any $v$ and thus  the last term in \eqref{herewe} is finite.
Hence the integrand on the lhs of \eqref{516} is the (Bochner integrable) function (a.e. $q_v^\prime(t)$) we were looking for and further we have by \cite[Proposition 1.2.2]{abhn} that $q_v(t)$ is also $AC\l [0,T]; \mathfrak{B} \r$ and a.e. differentiable.
\end{proof}
Before stating the governing equation we need the following techical Lemma.
\begin{lem}
\label{lemmaexist}
Under the assumptions of Lemmas \ref{qindomg}, \ref{lemmadiff} we have that the function $[0,T] \ni t \mapsto q_v^0(t):= q_v(t)-T_v(t)u(v)$ is in $H^v$ (see \eqref{defhv}), for any $v \in S$.
\end{lem} 
\begin{proof}
Note that $q_v^0(0) = \bm{0}$ and further
\begin{align}
q_v^0(t) \, = \, &u(v) + \int_0^t q^\prime_v(s) ds - \l u(v)+ \int_0^t T_v(s)G_vu(v) ds \r \notag \\
= \, & \int_0^t q^\prime_v(s) ds - \int_0^t T_v(s)G_vu(v) ds
\end{align}
where we used Lemma \eqref{lemmadiff} and the well known integral representation of semigroups for $u(v) \in \mfB_0$. Then since we know from Lemma \ref{qindomg} that $G_vq(t) \in L^1 \l [0,T]; \mathfrak{B} \r$ it follows from Lemma \ref{lemmadt} that $q_v^0 \in H^v$.
\end{proof}

\subsection{The governing equation}
In this section we give the rigorous result on the governing equation of $q_v(t)$. 
\begin{te}
\label{teeqf}
Let the assumptions of Lemma \ref{lemmaexist} prevail. Then $q_v(t)$ satisfies the following problem on $\mfB$
\begin{align}
f \l \partial_t-G_v \r \l q_v(t)-T_v(t)q_v(0) \r \, = \, \theta_v \int_{S} \l q_{v^\prime}(t)- q_v(t) \r h_v(dv^\prime), \qquad q_v(0) = u(v),
\label{eqtee}
\end{align}
for any $v \in S$ and $t \geq 0$.
\end{te}
\begin{proof}
Use \cite[Theorem 3.1]{koro} to say that
\begin{align}
q_v(t) \, =  \, &\mathds{E}^v \mathcal{T}(t)u(V(t))\mathds{1}_{\ll J_1 > t \rr}+ \mathds{E}^v \mathcal{T}(t)u(V(t))\mathds{1}_{\ll J_1 \leq t \rr} \notag 	\\	 
= \, &T_v(t)u(v) P^v(J_1>t)+  \int _0^t  \int _S T_v(s)q_{v^\prime}(t-s)\, h_v(dv^\prime) \,  P^v \l J_1\in ds\r.
\label{this}
\end{align}
Since $\left\| u(V(t)) \right\| \leq \sup_v \left\| u(v) \right\|< \infty$ we have that $u(V(t))$ is $P^v$-Bochner integrable and since $\left\| \mathcal{T}(t)u(V(t))\mathds{1}_{\ll J_1 \leq t \rr} \right\| \leq \left\|u(V(t)) \right\|$ we have that $ \mathcal{T}(t)u(V(t))\mathds{1}_{\ll J_1 \leq t \rr}$ is $P^v$-Bochner integrable. It follows that also $T_v(-t)\mathcal{T}(t)u(V(t))\mathds{1}_{\ll J_1 \leq t \rr}$ is $P^v$-Bochner integrable and thus we can apply $T_v(-t)$ to both sides of \eqref{this} and move $T_v(-t)$ inside the integral to get
\begin{align}
T_v(-t)q_v(t)= u(v) P^v(J_1>t)+  \int_0^t   \int _S  T_v(-(t-s))q_{v^\prime} (t-s) \,  h_v(dv') \, P^v(J_1\in ds).
\end{align}
Introduce the notation $\varphi_{v, v^\prime}(t):=T_{v}(-t)q_{v^\prime}(t)$. We have that
\begin{align}
\varphi_{v, v}(t)= u(v) P^v(J_1>t)+  \int_0^t   \int _S  \varphi_{v, v^\prime}(t-s) \,  h_v(dv') \, P^v(J_1\in ds).
\label{take}
\end{align}
Since $q_v \in L^1 \l [0,T]; \mfB \r$ for any $T>0$, it follows that $q_v \in L^1_{\text{loc}} \l \mathbb{R}^+; \mfB \r$ and thus we can take the Laplace transform $(t \mapsto \lambda)$ in both sides of \eqref{take} in the sense of \cite[chapter 1.4]{abhn}. As the last term in \eqref{take} is a convolution, we can use \cite[Proposition 1.6.4]{abhn} to get that, for $\lambda >0$,
\begin{align}
\widetilde{\varphi}_{v,v}(\lambda)= \frac{f(\lambda)}{\lambda} \frac{1}{\theta_v +f(\lambda)}u(v) +  \frac{\theta_v}{\theta_v +f(\lambda)} \int _S  \widetilde{\varphi}_{v,v^\prime}(\lambda) h_v(dv').
\label{lapl1}
\end{align}
For the Laplace transform of $P^v \l J_1 \in ds \r$ and $P^v \l J_1 > t \r$ see \cite[eqs (4.5) and (4.6)]{meerpoisson}.

Now multiply by $\lambda^{-1}$ both sides of \eqref{lapl1} and rearrange to get
\begin{align}
\frac{f(\lambda)}{\lambda}  \widetilde{\varphi}_{v,v}(\lambda)   -  \frac{f(\lambda)}{\lambda^2} u(v) = \frac{1}{\lambda} \theta_v\int _S  \l  \widetilde{\varphi}_{v,v^\prime}(\lambda) -\widetilde{\varphi}_{v, v}(\lambda) \r h_v(dv').
\label{laplspace}
\end{align}
Let $\varphi^0(w):= \varphi_{v,v}(w)-u(v)$. Then the following equation
\begin{align}
-&\int_0^t \int_0^\infty \l \Gamma_s\varphi^0(w)- \varphi^0(w) \r \nu (ds)dw \, = \, \int_0^t \theta_v \int_S \l \varphi_{v,v^\prime}(w) - \varphi_{v,v}(w)  \r \, h_v(dv^\prime) \, dw,
\label{534}
\end{align}
where $\Gamma_s$ is the operator defined in \eqref{defgammagen},
coincides in the Laplace space $(t \mapsto \lambda)$ with \eqref{laplspace} and this proves that \eqref{534} is verified for almost all $t \geq 0$. First note that the lhs is an element of $\mathfrak{B}$ since the integrand is such that
\begin{align}
\int_0^\infty \l \Gamma_s\varphi^0(w)- \varphi^0(w) \r \nu (ds) \, = \,  \int_0^\infty T_v(-w) \l U_s^vq^0_v(w)- q^0_v(w) \r \nu (ds)
\label{353}
\end{align}
where $q_v^0(w):=q_v(w) - T_v(w)u(v)= T_v(w) \varphi^0(w)$, and the rhs of \eqref{353} is in $L^1\l [0,T]; \mathfrak{B} \r$ by Lemma \ref{lemmaexist}.
For the rhs note that
\begin{align}
 \theta_v \int_{S}  \left\| \l q_{v^\prime}(t)- q_v(t) \r \right\| h_v(dv^\prime)  \, \leq \, 2 \theta_v \sup_v\left\| q_v(t) \right\| \, \leq \, 2  \theta_v \sup_v \left\| u(v) \right\| < \infty.
\end{align} 
For checking the Laplace transform of the rhs one just need to apply Fubini theorem for Bochner integral (e.g. \cite[Theorem 1.1.9]{abhn}) together with \cite[Corollary 1.6.5]{abhn}. For the lhs \revbt{the existence of Laplace transform can be ascertained by standard arguments (e.g., \cite[Theorem 13.6]{librobern})
%\begin{align}
%\int_0^T \left\|e^{-\lambda w} \int_0^\infty \l \Gamma_s \varphi^0(w)-\varphi^0(w)\r  \nu (ds) \,\right\| \, dw 
%%\notag \\ = \,& \int_0^T \left\|e^{-\lambda w} T_v(-w) \int_0^\infty \l U_s^vq_v^0(w)-q_v^0(w) \r  \nu (ds) \,\right\| \, dw \notag \\
% \leq \,& \int_0^T e^{-\lambda w} \left\| A_v^U q_v^0(w) \right\|d w \int_0^1 s \nu(ds) + 2\nu[1, +\infty) \int_0^T \left\| q_v^0(w) \right\| dw
%\end{align}
using the definition of $A^U_v$ together with the estimates in \eqref{331}, \eqref{herewe} and \eqref{stimamediagg}. Then} we compute 
\begin{align}
&-\int_0^\infty e^{-\lambda t} \int_0^t \int_0^\infty \l \Gamma_s\varphi^0 (w) - \varphi^0(w) \r \nu (ds) \, dw \, dt \notag \\ = \, &- \frac{1}{\lambda} \int_0^\infty e^{-\lambda w} \int_0^\infty \l \Gamma_s\varphi^0 (w) - \varphi^0(w)  \r \nu (ds) \, dw \notag \\
= \, &  -\frac{1}{\lambda} \int_0^\infty e^{-\lambda w} \int_0^\infty \l \varphi_{v,v} (w-s) \mathds{1}_{\ll s \leq w \rr} - \varphi_{v,v} (0) \mathds{1}_{\ll s\leq w \rr} -  	 \varphi_{v,v} (w) + \varphi_{v,v}(0) \r \nu(ds) \, dw \notag \\
= \, & -\frac{1}{\lambda} \int_0^\infty e^{-\lambda w} \int_0^\infty \l \varphi_{v,v} (w-s) \mathds{1}_{\ll s \leq w \rr} + \varphi_{v,v} (0) \mathds{1}_{\ll s>w \rr} - \varphi_{v,v} (w) \r \nu(ds) \, dw  \notag \\
= \, & -\frac{1}{\lambda}  \int_0^\infty \l e^{-\lambda s} \widetilde{\varphi}_{v,v}(\lambda) + \lambda^{-1}\varphi_{v,v}(0) \l 1 -e^{-\lambda s}\r - \widetilde{\varphi}_{v,v}(\lambda) \r \nu(ds) \notag \\
= \, & - \frac{1}{\lambda} \int_0^\infty \l \l \widetilde \varphi_{v,v}(\lambda)-\lambda^{-1}  \varphi_{v,v}(0)\r \l e^{-\lambda s}-1 \r \r \nu(ds) \notag \\
= \, & \frac{f(\lambda)}{\lambda} \widetilde \varphi_{v,v}(\lambda) \, - \, \frac{f(\lambda)}{\lambda^2} u(v)
\end{align}
where in the last step we used \eqref{reprbern}, in the third last we used Fubini while in the second line we used again \cite[Corollary 1.6.5]{abhn}. 

 The fact that the equality \eqref{534} is true for any $t \geq 0$ (and not only for almost all $t \geq 0$) comes from (strong) continuity of both sides which is a consequence of properties of Bochner integrals (e.g. \cite[Proposition 1.2.2]{abhn}). Further $s \mapsto \varphi_{v,v^\prime}(s)-\varphi_{v,v}(s) = T_v(-s) \l q_{v^\prime}(s)-q_v(s) \r $ and, since $s \mapsto q_v(s)$ is continuous by Lemma \ref{lemmadiff}, we have that $s \mapsto \varphi_{v,v^\prime}(s)-\varphi_{v,v}(s)$ is continuous. Thus we can take the strong derivative in \eqref{534}, use representation \eqref{353} and apply $T_v(t)$ to both sides of the equation to get
\begin{align}
f \l \partial_t-G_v \r (q_v(t)-T_v(t)q_v(0))  \, = \, \theta_v \int_{S} \l q_{v^\prime}(t)- q_v(t) \r h_v(dv^\prime).
\end{align}
\end{proof}
\begin{os} \normalfont
Take a function $u \in L^1 \l [0,T]; \mfB \r$ with $u (0) \neq \bm{0}$ so that $u \notin H^v$ and define $u^0(t):=u(t) - T_v(t) u(0)$. Now $u^0(0)=\bm{0}$ so that one could have $u^0 \in H^v$. This is exactly the case of $q_v$ and $q_v^0$. Here the regularizing term $-T_v(t) u(0)$ has the same role of the regularization $-u(0)$ in the canonical fractional derivative \eqref{spieg}.
Indeed, whenever $u(t)$ is such that following computations are justified, the function $f \l \partial_t - G_v \r u^0(t)$ can be rewritten as
\begin{align}
&f\l \partial_t - G_v) \r u^0(t)\notag \\ = \,&- \int_0^\infty \left[U_s^vu^0(t)-u^0(t) \right] \nu(ds)\notag \\
= \, & - \int_0^\infty \left[U_s^vu(t)-U_s^vT_v(t)u(0)-u(t)+T_v(t) u(0) \right] \nu(ds).
\end{align}
Since
\begin{align}
U_s^vT_v(t)u(0) \, = \, &T_v(s) \Gamma_s T_v(t)u(0) \notag \\
= \, & T_v(s) T_v(t-s) u(0) \mathds{1}_{[s \leq t]} \notag \\
= \, & T_v(t)  u(0) \mathds{1}_{[s\leq t]}
\end{align}
it follows that
\begin{align}
f\l \partial_t - G_v) \r u^0(t) \, = \,-  \int_0^\infty  \left[ U_s^vu(t) - u(t) \right] \nu(ds) -\bar{\nu}(t) T_v(t)  u(0)
\label{316}
\end{align}
which has the form of the first line in \eqref{spieg}. It follows that the lhs in \eqref{eqtee} has the structure of a canonical fractional derivative.
A particular case which will be of great interest in the next sections is given by specializing $T_v(t)u(x)=u(x+vt)$ on suitable function spaces, to get from \eqref{316}
\begin{align}
 f \l \partial_t -v \cdot \nabla_x \r (u(t,x)-u(0,x+vt))\,  = \,  &\int_0^\infty \l u(t-s,x+vs) \mathds{1}_{[s \leq t]} - u(t,x) \r \nu (ds) \notag \\
& - \,\bar{\nu}(t) u(0, x+vt).
\label{358}
\end{align}
In this last case our operator provides a rigorous way to define explicitly $\l \partial_t - v \cdot \nabla_x  \r^\alpha$, which is usually understood as a pseudo-differential operator (see Section \ref{secpl} below for details and references).
\end{os}

We proved until now that the function $q_v(t)$ satisfies the Boltzmann-type equation \eqref{eqtee} where the operator $f \l \partial_t-G_v \r$ is obtained with Bochner subordination theory. This operator in the governing equation represents the ``coupling'' between the time evolution (delayed by the inverse subordinator) and the evolution on $\mfB$ (characterized by the non-exponential waiting times), induced by the non Markovian perturbations. Now we show that the function 
\begin{align}
\varphi_{v,v^\prime}(t) \, : = \, \mathds{E}^{v^\prime} T_v{(-t)} \mathcal{T}(t)u(V(t)) \, = \, T_v(-t) q_{v^\prime}(t)
\label{defvp}
\end{align}
satisfies an integro-differential equation of fractional type. We remark that applying $T_v(-t)$ to the random operator $\mathcal{T}(t)$ makes, under $P^v$, the evolution on $\mfB$ constant before the first perturbation induced by the semi-Markov process: 
\begin{align}
T_v(-t)\mathcal{T}(t) u(V(t)) \mathds{1}_{[t \leq \tau_1]} \, = \, u(v) .
\end{align}
Therefore we will see that the equation contains only the scattering component and not the operator $G_v$.

In order to write the equation, let us first recall the canonical form of the fractional derivative (said Caputo-type),
\begin{align}
\mathcal{D}_t^\alpha \phi(t) \, = \,& \partial_t \int_0^t \phi(s) \frac{(t-s)^{-\alpha}}{\Gamma (1-\alpha)} ds - \phi(0) \frac{t^{-\alpha}}{\Gamma (1-\alpha)} \notag \\
= \, & \partial_t \int_0^t \l \phi(s)- \phi(0) \r \frac{(t-s)^{-\alpha}}{\Gamma (1-\alpha)}ds,
\end{align}
which can be generalized by replacing the kernel $s^{-\alpha}/\Gamma (1-\alpha)$ with the tail of the L\'evy measure $\bar{\nu}(s)$ as
\begin{align}
\mathcal{D}_t^f \phi(t) \, := \,& \partial_t \int_0^t \phi(s) \bar{\nu}(t-s) ds - \phi(0) \bar{\nu}(t) \notag \\
= \, & \partial_t \int_0^t \l \phi(s)- \phi(0) \r \bar{\nu}(t-s) ds.
\end{align}
for suitable functions $\phi$. Here is the rigorous statement on the equation.

\begin{prop}
If $\sup_v \left\| u(v) \right\|<\infty$, the function $t \mapsto \varphi_{v,v^\prime}(t)$ defined in \eqref{defvp} satisfies the following problem on $\mfB$
\begin{align}
\mathcal{D}_t^f \varphi_{v,v}(t) \, = \,  \theta_v \int_S \l \varphi_{v,v^\prime}(t) - \varphi_{v,v}(t)  \r \, h_v(dv^\prime), \qquad \varphi_{v,v^\prime}(0) = u(v^\prime),
\label{eqvecchiate}
\end{align}
for any $v, v^\prime \in S$ and $t \geq 0$.
\end{prop}
\begin{proof}
Take the representation in the Laplace space given in \eqref{laplspace}
\begin{align}
\frac{f(\lambda)}{\lambda}  \widetilde{\varphi}_{v,v}(\lambda)   -  \frac{f(\lambda)}{\lambda^2} u(v) = \frac{1}{\lambda} \theta_v\int _S  \l  \widetilde{\varphi}_{v,v^\prime}(\lambda) -\widetilde{\varphi}_{v, v}(\lambda) \r h_v(dv').
\label{laplspace2}
\end{align}
Note that \eqref{laplspace2} is still valid here since it comes from the renewal equation \eqref{this}. In the same spirit of the Theorem \ref{teeqf} we first prove that the equation
\begin{align}
\int_0^t \l \varphi_{v,v}(s)- \varphi_{v,v}(0) \r \, \bar{\nu}(t-s) ds \, = \, \int_0^t \theta_v \int_S \l \varphi_{v,v^\prime}(w)  - \varphi_{v,v}(w) \r \, h_v(dv^\prime) dw
\label{intvecchia}
\end{align}
coincides in Laplace space $(t \mapsto \lambda)$ with \eqref{laplspace2}. For the right-hand side we already did this in the proof of Theorem \ref{teeqf}. For the left-hand side first note that since $t \mapsto \varphi_{v,v}(t)$ is in $L^1 \l [0,T]; \mathfrak{B} \r$ for any $T>0$ and $t \mapsto \bar{\nu}(t)$ is in $L^1 \l[0,T]; \mathbb{R}\r$ for any $T>0$, then their Laplace convolution is in $L^1\l [0,T]; \mfB \r$ (see the discussion at the end of \cite[page 22]{abhn}). Then, by an application of the convolution theorem for Laplace transform for Bochner integrals (e.g. \cite[Proposition 1.6.4]{abhn}), using that
\begin{align}
\int_0^\infty e^{-\lambda t} \bar{\nu}(t) dt \, = \, \frac{f(\lambda)}{\lambda},
\end{align}
the equality follows.
We already proved that the rhs of \eqref{intvecchia} is differentiable in Theorem \ref{teeqf}. The lhs is instead continuous since, using that $\bar{\nu}(s) \in L^1\l [0,T], \mathbb{R} \r$, for any $T>0$, and $\left\| \varphi_{v,v}(t) \right\| \leq \sup_v \left\| u(v) \right\|$, we can apply \cite[Proposition 1.3.2 (and the discussion below)]{abhn}. Hence both sides of \eqref{intvecchia} can be differentiated to get, for any $t >0$,
\begin{align}
\mathcal{D}_t^f \varphi_{v,v}(t) \, = \, \theta_v\int_S \l \varphi_{v,v^\prime}(t)  - \varphi_{v,v}(t) \r \, h_v(dv^\prime).
\label{finito}
\end{align}
\end{proof}

\section{Abstract wave equation with semi-Markov damping}
\label{secdw}
We find here a generalized wave equation with damping which governs the expected value of a particular class of random evolutions. The results in this section provide the semi-Markov counterpart of the theory valid in the Markov case, given in \cite[Section 4]{griegohersh}.

In this section we still work under A1) and A2). However, we further assume that:
\begin{enumerate}
\item $\mathcal{S}= \{ v_1, v_2\}$
\item $\theta_{v_1} = \theta_{v_2}=:\theta$, so that $N(t)$ is a renewal counting process.
\item The semigroups $T_{v_1}(t)$ and $T_{v_2}(t)$ commute.
\item The generators $G_{v_1}, G_{v_2}$ are scalar multiples of the same operator, with  the form $G_{v_1}=G$ and $G_{v_2}=-G$, where $G$ generates the group $T(t), t \in \mathbb{R}$.
\end{enumerate}
Assumption $(3)$ means that    $G_{v_1} $ generates $T_{v_1}(t) =T(t)$, $t \geq 0$,  the forward evolution for $T$, while $G_{v_2}=-G$ generates $T_{v_2}(t)=T(-t)$, $t \geq 0$, the backward evolution for $T$.
As a remarkable example, the above hypoteses include the translation group on $C_0(\mathbb{R})$ or $L^1 \l \mathbb{R} \r$,  such that $T(t)u(x)= u(x+t)$ and $G=d/dx$ (in suitable sense).

Without losing the generality of the above hypotheses, we will assume $\mathcal{S}=\{1,-1\}$. Hence, the underlying semi-Markov process can be written in the ``telegraph'' form 
\begin{align}
V(t)= V_0 (-1)^{N(t)} \label{processo del telegrafo}
\end{align}
where $V_0$ is a random variable with values in $\mathcal{S}$ and $N(t)$ is the number of renewals up to time $t$.

As before use $\mathcal{T}(t)$ for the random evolution operator and define
\begin{align*} 
q_{1}(t):= \mathds{E}^1\bigl ( \mathcal{T}(t) u(V(t)) \bigr ) \qquad q_{-1}(t):= \mathds{E}^{-1}\bigl ( \mathcal{T}(t) u(V(t))\bigr ).
\end{align*}
If $V_0$ is a r.v. with $P(V_0=1)=P(V_0=-1)=1/2$, hence
\begin{align}
q(t) \, := \,& \mathds{E} \mathds{E}^{V_0}(\mathcal{T}(t)u(V(t))  ) \notag\\
= \,  & q_1(t)P(V_0=1)+q_{-1} (t)P(V_0=-1) \notag \\
= \, & \frac{q_1(t)+q_{-1}(t)}{2}.\label{media}
\end{align}
In \cite{griegohersh} the authors studied the special case of Markov random evolutions and found a second order abstract equation. By using our notations, we re-formulate their result in the following theorem (for the proof we refer to the original paper). 
\begin{te}\cite[Section 4]{griegohersh}
Let  $u(1)=u(-1):=u$, so that  $u(V(t))=u$ for any $t\geq 0$, where $u\in Dom (G^2)$.
Moreover, let  $P(V_0=1)= P(V_0=-1)=1/2$ and let $\mathcal{N}(t)$ be a Poisson process with intensity $\theta >0$. Then \eqref{media} solves the abstract telegraph equation
\begin{align}
\l \partial_t^2 -G^2 \r q(t)=-2 \theta \partial_t q(t) \label{Damped wave equation},
\end{align}
under the initial conditions $q(0)=u$ and $q'(0)=0$.
\end{te}
Equation \eqref{Damped wave equation} is also called abstract wave equation with damping. The reason of this name is that, in the case of translation groups on $C_0(\mathbb{R})$, where $G^2= d^2/dx^2$, the abstract equation \eqref{Damped wave equation} is formally equivalent to the classical damped wave equation \eqref{damped markov} as we will see in the next section (the equivalence is only formal because \eqref{Damped wave equation} is  an abstract equation while \eqref{damped markov} is a classical equation).

We observe that (\ref{Damped wave equation}) has an interesting connection 
  to the abstract wave equation
\begin{align}
\partial^2_tw(t)=G^2w(t) \qquad w(0)=u \qquad w'(0)=0,\label{equazione onde}
\end{align}
where $u\in Dom (G^2)$.
It is immediate to verify that  (\ref{equazione onde})  is solved by  the ``free'' evolution 
\begin{align}
w(t)= \frac{1}{2} \bigl ( T(t)u+T(-t)u \bigr ). \label{free evolution}
\end{align}
It is proved by Griego and Hersh \cite{griegohersh} that the solution to \eqref{Damped wave equation} is related with $w(t)$ by the following equality
\begin{align}
q(t) \, = \, \mathds{E} w \l \int_0^t (-1)^{\mathcal{N}(\tau)}  d\tau \r
\end{align}
where $\mathcal{N}(\tau)$ is the underlying Poisson process pacing the jumps of the velocity.
Heuristically, this means that the solution to \eqref{Damped wave equation} is obtained by perturbing the free evolution  \eqref{free evolution} at random times. This causes the origin of the damping term  in  eq. (\ref{Damped wave equation}), making eq.  (\ref{Damped wave equation}) different from eq. (\ref{equazione onde}). This is proved in \cite[Thm. 4]{griegohersh}.
 
In the following Theorem, we show that $q(t)$ in \eqref{media}, also in our general semi-Markov setting, is again the average of the free evolution \eqref{free evolution} with the time $t$ suitably randomized.

\begin{te}
Let $q(t)$ and $w(t)$ be defined in  (\ref{media})       and  (\ref{free evolution}). The following representation holds:
 $$q(t)=\mathds{E}\, w\biggl (\int _0^t (-1)^{N(\tau)}d\tau \biggr ),$$
 where $N(\tau)$ is the underlying renewal counting process.
\end{te}

\begin{proof}
Let us denote the occupation measure of $V(t)$ in the state $v$ by
\begin{align}
H_t^v \, = \, \int_0^t \mathds{1}_{\ll V(s) = v \rr}ds\rev{,} \qquad v=\pm1\rev{,}
\end{align} 
and observe that the difference of occupation times can be written as
\begin{align}
H_t^1 -H_t^{-1}=V_0 \int _0^t (-1)^{N(\tau)}d\tau.
\end{align}
Taking into account that $T_1(t)$ and $T_{-1}(t)$ commute, and using the semigroup property, the random evolution can be re-written as
\begin{align}
\mathcal{T}(t)u = T_1 (H_t^1)\, T_{-1}(H_t^{-1})u \label{r1}
\end{align}
so that the waiting times $J_i$ are no more relevant.     Using the assumption that $T_{\pm 1}(t)=T(\pm t)$, where $T(t)$ is a group, we can write $(\ref{r1})$ as 
\begin{align}
\mathcal{T}(t)u = T (H_t^1)\, T(-H_t^{-1})u =T( H_t^1-H_t^{-1})u = T\biggl (V_0 \int _0^t (-1)^{N(\tau)}d\tau \biggr )u  .  \label{r2}
\end{align}
Let $\gamma _t=\int _0^t (-1)^{N(\tau)}d\tau$.  The expected value of (\ref{r2}) can be written as
\begin{align}
q(t)= \mathds{E} (\mathcal{T}(t)u) &= \frac{1}{2} \biggl ( \mathds{E} (\mathcal{T}(t)u |V_0=1)  +\mathds{E} (\mathcal{T}(t)u|V_0=-1) \biggr ) \notag \\
&= \frac{1}{2}  \biggl (    \mathds{E} \bigl (T (\gamma _t)   u \bigr  )  +\mathds{E} \bigl ( T(-\gamma _t)u         \bigr )   \biggr )   \notag   \\
& = \frac{1}{2}\mathds{E} \biggl ( T(\gamma _t)u + T(-\gamma _t) u  \biggr ) \notag \\
&= \mathds{E}w(\gamma _t)
\end{align}
and the proof is complete.
\end{proof}
We derive now the semi-Markov version of the damped wave equation \eqref{Damped wave equation}. Let us first note that the d'Alembert-type operator on the lhs of \eqref{Damped wave equation} is, formally,
\begin{align}
\square:=\l \partial_t^2-G^2 \r \, = \, \l \partial_t -G \r \l \partial_t +G \r.
\end{align}
In order to write down the governing equation of $q(t)$ in our semi-Markov setting we introduce the compact notation
\begin{align}
&\mathcal{D}^+ :=f \l \partial_t +G \r \label{413}  \\
& \mathcal{D}^- := f \l \partial_t -G \r \label{414}
\end{align}
where the operators on the rhs of \eqref{413} and \eqref{414} are defined as in Section \ref{secdeff}.
Then we define the operator
\begin{align}
\mathcal{D}_{\square} := \mathcal{D}^+ \mathcal{D}^- 
\label{55}
\end{align}
on $L^1 \l [0,T];\mfB \r$ with domain $D \subset L^1 \l [0,T];\mfB \r$.

We observe that when $f(\lambda) = \lambda^\alpha$, the operator in \eqref{55} has the form of the fractional power
\begin{align}
\square^\alpha := \l \partial_t^2 -G^2 \r^\alpha.
\end{align}

Let us first note that, under the more restrictive assumptions in this section, the sets $H^v$, $v \in \ll  -1,1 \rr$ defined in \eqref{defhv} are the same set for any $v$. Hence we will use the notation $H$.

\begin{lem}
\label{commute}
The operators $\mathcal{D}^-$ and $\mathcal{D}^+$ commute on $H \cap D$.
\end{lem}
\begin{proof}
Note that
\begin{align}
f \l \partial_t -G \r f \l \partial_t+G \r \phi (t) \,
= \, &\int _0^\infty \int _0^\infty \biggl [  \phi (t)-T(s')\phi (t-s')\mathds{1}_{[s'<t]}-T(t)\phi (0) \mathds{1}_{[s'\geq t]} \notag\\   
& -T(-s)\phi  (t-s)\mathds {1}_{[s<t]} +\mathds{1}_{[s<t]}\mathds{1}_{[s'<t-s]}T(s'-s) \phi (t-s-s') \notag \\
&  +\mathds{1}_{[s<t]} \mathds{1}_{[s'\geq t-s]} T(t-2s)\phi (0)\biggr ] \nu (ds)\nu(ds').\label{integrale doppio} 
\end{align}
Since $\phi \in H$, it follows that $\phi (0) = \bm{0}$ and therefore we have from \eqref{integrale doppio}
\begin{align}
 f \l \partial_t -G \r f \l \partial_t+G \r \phi (t) = \, & \int _0^\infty \int _0^\infty \biggl [  \phi (t)-T(s')\phi (t-s')\mathds{1}_{[s'<t]}  
 -T(-s)\phi  (t-s)\mathds {1}_{[s<t]} \notag \\
 & + \mathds{1}_{[s<t]}\mathds{1}_{[s'<t-s]}T(s'-s) \phi (t-s-s') \bigg] \nu (ds)\nu(ds')  \notag \\
 = \, & f \l \partial_t +G \r f \l \partial_t-G \r \phi(t)
\end{align}
where exhange in the order of integrals 
is guaranteed by Fubini Theorem for double Bochner integrals (\cite[Thm. 1.1.9]{abhn}), since $\phi \in D$.
\end{proof}

We are now in position to state the following result, which is a semi-Markov counterpart of eq. (\ref{Damped wave equation}). Before stating the following theorem we introduce the functions $w^{\pm}: [0, \infty) \mapsto \mfB$ defined as
\begin{align}
w^{\pm}(t):= T(\pm t) u.
\end{align}
We remark that $w^{\pm}$ are the unique solutions to the wave equation, under suitable initial conditions:
\begin{align}
\l \partial_t^2 - G^2 \r \phi(t) = 0, \qquad \phi(0) = u, \qquad \phi^\prime (0) = \pm Gu.  
\end{align}

\begin{te}  \label{teorema finale}
Let  $u(1)=u(-1):=u$, so that  $u(V(t))=u$ for any $t\geq  0$, where $u\in Dom (G^2)$.
Moreover, let  $P(V_0=1)= P(V_0=-1)=1/2$. Then
\eqref{media}
satisfies the equation on $\mfB$
\begin{align} \label{Equazione secondo ordine}
\mathcal{D}_\square & \l q(t)- w(t) \r \, = \,  -\theta \mathcal{D}^-  \l q(t)- w^+(t) \r - \theta \mathcal{D}^+ \l q(t)-w^-(t)\r,
\end{align}
under the initial condition $q(0)=u$, where $w^{\pm}(t) = T(\pm t) q(0)$ and $w(t)$ are the free evolution (see \eqref{free evolution} for $w(t)$), i.e., the unique solutions to the abstract wave equation \eqref{equazione onde} under appropriate initial conditions, according to the above definitions.
\end{te}

\begin{proof}

The equations given in Theorem \ref{teeqf} split into
\begin{align}
&\mathcal{D}^- \l q_{1}(t) -T(t)q_1(0) \r    = -\theta  q_{1}(t)+\theta  q_{-1}(t) ,\label{Z11} \\      
&\mathcal{D}^+\l q_{-1}(t) -T(-t)q_{-1}(0) \r =  \theta  q_{1}(t)-\theta  q_{-1}(t) . \label{Z22}
\end{align}
under $q_{1}(0)=q_{-1}(0)=u$.
By adding and subtracting the appropriate terms we have
\begin{align}
&\mathcal{D}^- \l q_{1}(t) -T(t)q_1(0) \r    = -\theta  \l q_{1}(t) - T(t)q_1(0)\r+\theta  \l q_{-1}(t) - T(-t) q_{-1}(0) \r  \notag \\ & - \theta\l T(t) - T(-t) \r u , \label{521}\\      
&\mathcal{D}^+\l q_{-1}(t) -T(-t)q_{-1}(0) \r =  \theta \l  q_{1}(t) - T(t)q_1(0)\r-\theta  \l q_{-1}(t)-T(-t)q_{-1}(0)\r \notag \\
& + \theta \l T(t) - T(-t) \r u . \label{522}
\end{align}
By Lemma \ref{lemmaexist} we know that both $q_1 (t)-T(t)q_1(0)$ and $q_{-1}(t)-T(-t)q_{-1}(0)$ are in $H$ and therefore they lie in the domain of $\mathcal{D}^+$ and $\mathcal{D}^-$ and further, by linearity, also their linear combination does. Further, since $u \in \mfB_0$ it is clear that $\l T(t)-T(-t) \r u \in H$ and thus by Lemma \ref{domfg} it also lies in the domain of $\mathcal{D}^+$. So we can apply  $\mathcal{D}^+$ to both sides of \eqref{521}. By analogous considerations, we can apply $\mathcal{D}^-$ to both sides of \eqref{522}. It follows that $\l q_1 (t)-T(t)q_1(0) \r$ and $q_{-1}(t)-T(-t)q_{-1}(0)$ are elements of $D \cap H$.
Hence we can apply Lemma \ref{commute} to get
\begin{align}
&\mathcal{D}_{\square} \l q_{1}(t) -T(t)q_1(0) \r    = -\theta \mathcal{D}^+  \l q_{1}(t) - T(t)q_1(0)\r+\theta \mathcal{D}^+ \l q_{-1}(t) - T(-t) q_{-1}(0) \r\notag \\ & -\theta\mathcal{D}^+\l T(t) - T(-t) \r u ,\label{426}\\      
&\mathcal{D}_{\square}\l q_{-1}(t) -T(-t)q_{-1}(0) \r =  \theta \mathcal{D}^-\l  q_{1}(t) - T(t)q_1(0)\r-\theta  \mathcal{D}^-\l q_{-1}(t)-T(-t)q_{-1}(0)\r \label{427}\notag \\
& +\theta\mathcal{D}^- \l T(t) - T(-t) \r u . 
\end{align}
Using now \eqref{521} and \eqref{522} we see that
\begin{align}
\mathcal{D}^+ \l q_{-1}(t) - T(-t) q_{-1}(0) \r \, = \, -\mathcal{D}^- \l q_{1}(t) - T(t) q_{1}(0) \r
\end{align}
and substituting in \eqref{426} and \eqref{427} we obtain
\begin{align}
&\mathcal{D}_{\square} \l q_{1}(t) -T(t)q_1(0) \r    = -\theta \mathcal{D}^+  \l q_{1}(t) - T(t)q_1(0)\r-\theta \mathcal{D}^- \l q_{1}(t) - T(t) q_{1}(0) \r\notag \\ & - \theta\mathcal{D}^+\l T(t) - T(-t) \r u , \\      
&\mathcal{D}_{\square}\l q_{-1}(t) -T(-t)q_{-1}(0) \r = - \theta \mathcal{D}^+\l  q_{-1}(t) - T(-t)q_{-1}(0)\r-\theta  \mathcal{D}^-\l q_{-1}(t)-T(-t)q_{-1}(0)\r \notag \\
& +\theta\mathcal{D}^-  \l T(t) - T(-t) \r u ,
\end{align}
and again by linearity
\begin{align}
&\mathcal{D}_{\square} \l q_{1}(t) -T(t)q_1(0) \r    = -\theta \mathcal{D}^+  \l q_{1}(t) - T(-t)u\r-\theta \mathcal{D}^- \l q_{1}(t) - T(t) u \r \label{sum1} \\ 
&\mathcal{D}_{\square}\l q_{-1}(t) -T(-t)q_{-1}(0) \r = - \theta \mathcal{D}^+\l  q_{-1}(t) - T(-t)u\r-\theta  \mathcal{D}^-\l q_{-1}(t)-T(t)u\r. \label{sum2}
\end{align}
By summing \eqref{sum1} and \eqref{sum2} we get
\begin{align}
\mathcal{D}_\square\l 2q(t) - (T(t)+T(-t))u \r \, = \, -\theta \mathcal{D}^+ \l 2q(t) - 2 T(-t)u \r - \theta \mathcal{D}^- \l 2q(t) -2 T(t)u \r
\label{div2}
\end{align}
and then dividing by 2 both sides of \eqref{div2} we obtain the result.
\end{proof}

\section{Transport with infinite mean flight times and superdiffusion}

We here consider a model for scattering transport such that the waiting times between velocity changes (collisions) have infinite expectation. Therefore this model substantially differs from the classical Markov case with exponentially distributed time intervals, in particular for what concerns scaling limits. It is a classical result (e.g. \cite{Watanabe}) that a transport process with uniform velocity changes paced by i.i.d. exponential r.v.'s converges, after a suitable scaling limit, to a Brownian motion, therefore exhibiting a diffusive behaviour (the mean square displacement grows like $Ct$, $C>0$, as $t \to \infty$) and infinite velocity. However the exponential waiting time is not crucial to have this behaviour: if one takes arbitrary finite mean waiting times, then in the long run (or after a scaling limit) the convergence is still to a diffusive process (see, for example, \cite[chapter 3]{koro}). However the infinite expectation case seems to be always ruled out by classical assumptions in this literature. In this section we are able to deal with the infinite expectation case using some CTRWs scaling limit theory. In particular we show that a suitable scaling yields to a superdiffusive transport process whose one-dimensional distribution is supported (when the process starts from the origin) on
\begin{align}
\mathpzc{B}^{d}_t := \ll z \in \mathbb{R}^d: \left\| z\right\|_d \leq t \rr.
\end{align}
This agrees with the intuition because the longer flight times tend to be (on average) longer (heavy tailed) than in the exponential (or finite mean) case and permits a space scaling as fast as the time scaling.

Here is a more rigorous discussion.

Consider a semi-Markov model of transporting particle in $\mathbb{R}^d$ \bt{as follows}. It is assumed that the particle originating at $x\in \mathbb{R}^d$ moves along the unit vector $v_1$ with constant velocity $1$, until it has a collision after a random waiting time $J_1$; then the particle moves (again with constant velocity $1$) along the unit direction $v_2$ for a random time $J_2$, and so on. Let the $v_i$ have uniform distribution on the unit sphere 
\begin{align}
S^{d-1}:= \ll v  \in \mathbb{R}^d : \left\| v \right\|_d=1  \rr,
\end{align}
where $\left\| \cdot \right\|_d$ stands for the euclidean norm, independently on the past history, and let the $J_i$ be i.i.d. random variables in $\mathbb{R}^+$.
Hence, let  $V(t)$ be a semi-Markov chain on $S^{d-1}$, representing the unit velocity vector of the moving particle, i.e.,
\begin{align*}
V(t)= v_n  \qquad      \tau _{n-1}    \leq  t < \tau _{n}, \qquad n \in \mathbb{N}
\end{align*}
where $\tau _n= \sum _{i=1}^{n} J_i$ denotes the time of the $n$-th scattering, with $\tau _0=0$.

Let, for $x \in \mathbb{R}^d$,
\begin{align}
X(t)= x+\int _0^ t V(\tau)d\tau \label{processo posizione}
\end{align}
 be a continuous additive functional of $V$, representing the position of the particle.
With the usual notation
\begin{align*}
N(t) = \max \{ n\geq 0: \tau _n\leq t\},
\end{align*}
the number of scatterings up to time $t$, we can re-write (\ref{processo posizione}) as 
\begin{align}
X(t)=  x+\sum _{i=1} ^{N(t)} J_iv_i + \l t-\tau _{N(t)} \r v_{N(t)+1} .
\label{processo posizione 2}
\end{align}

We call $(X(t), V(t))$ semi-Markov isotropic transport process.

For any $x\in \mathbb{R}^d$ and $v\in  S^{d-1}$ we denote the one-dimensional distribution of the process, started at $x$ and $v$, as
\begin{align}
p_t(dz, dw|x,v):= P(X(t)\in dz, V(t)\in dw| V(0)=v), \label{transition probability}
\end{align}
which is compactly supported on the set 
$$\{(z,w)\in \mathbb{R}^d\times \mathcal{S}^{d-1}:||z-x||_d\leq t, w\in \mathcal{S}^{d-1}\}.$$
Denoting by  $\mathds{E}^{v}$ the integration with respect to the measure (\ref{transition probability}), we define the mean value
\begin{align}
g(x,v,t)=  \mathds{E}^{v} h(X(t), V(t)), \label{semigruppo semi markov}
\end{align}
where $h\in C _0  (\mathbb{R}^d \times \mathcal{S}^{d-1})$. Clearly,  (\ref{semigruppo semi markov}) is nothing more that the mean value defined in (\ref{meanvalue}) 
with  reference to the particular random evolution
\begin{align}
T_{V(t)} (t- \tau _{N(t)})     \cdots T_{v_2}(J_2)  T_{v_1} (J_1)
\end{align}
where $T_{v_i}$ are translations groups on  $C_0(\mathbb{R}^d \times \mathcal{S}^{d-1})$ acting on the coordinate $x \in \mathbb{R}^d$, namely $T_v(t)h(x,v^\prime):= h(x+vt,v^\prime)$, for any $v,v^\prime \in \mathcal{S}^{d-1}$.

In the special case where the waiting times have exponential law with mean $1/\theta$,  then $V(t)$ is a continuous time Markov chain and also $(X(t), V(t))$ is jointly Markov. Such a process has been studied in several papers (e.g. \cite{monin, Watanabe}). In this case it is true that \eqref{semigruppo semi markov} defines  a strongly continuous contraction semigroup on $ C_0(\mathbb{R}^d \times \mathcal{S}^{d-1})$ equipped with the sup-norm (consult \cite[section 2]{Watanabe});  for all  $h\in C _0^1 (\mathbb{R}^d \times \mathcal{S}^{d-1})$ the following equation holds
\begin{align}
\partial_t g(x,v,t)= v\cdot \nabla _x g(x,v,t) + \theta \int _{S^{d-1}} (g(x,v',t)-g(x,v,t))\mu (dv'), \label{equazione coppia}
\end{align}
under the condition $g(x,v,0)= h(x,v)$, where $\mu(\cdot)$ is the uniform measure on $\mathcal{S}^{d-1}$. 
Eq. \eqref{equazione coppia}, also known as linear Boltzmann equation, is the backward equation for the Markov process $(X(t), V(t))$ and the operator acting on the right side is the infinitesimal generator.

In the framework of statistical physics, one of the things that makes this Markov  process important is the fact that it is an approximation of a diffusion process; indeed $X(t)$ converges to a Brownian motion by re-scaling the space variable as $x\to cx$ and the time variable as $t\to t c ^2$ and letting $c \to \infty$.  For a discussion on this point, consult, for example, \cite[sections 3 and 4]{Watanabe}.

We stress again that, alongside this classic result, there is another important fact, which perhaps is not so well known: a large class of semi-Markov transport processes shares the same asymptotic property which leads to a limit diffusion.  This means that, in the scaling limit of small and rapid jumps, the exponential distribution of waiting times (and the consequent Markovianity of the process) is not a crucial condition to have convergence to Brownian motion. Rather, the only thing that matters is that the waiting times have finite mean and variance. We refer to the next Remark (\ref{oss}) for a sketched proof, and to \cite[Section 4.3]{koro} for a more general discussion. The reader can consult also \cite{Papanicolaou} for a general theory concerning the limits of Markov transport processes. Moreover, for an example of a transport process whose flight times have finite mean but are not exponential, consult \cite{alessandro}, where the authors assume Dirichlet distributed flight times.

Thus, transport processes with infinite mean waiting times are cut  out from the above consideration. But yet, semi-Markov processes with such a property have proven to be a fundamental tool in statistical physics, especially in models of anomalous diffusions (see, for example, \cite{FCbook, Metzler} and references therein). In Section \ref{sectionconvergence}, we focus on this aspect and find new results in this direction using some CTRWs limit theory.

\subsection{Transport processes with heavy-tailed flight times}
\label{secpl}
We here consider a particular type of semi-Markov process $V(t)$ on $\mathcal{S}^{d-1}$, whose waiting times exhibit power-law decaying densities, with infinite mean and variance.

In order to construct such a process, we refer to the time-change assumption introduced in section \ref{assumptions}. Hence, assume that $\sigma ^{\alpha} (t)$ is a stable subordinator with index $\alpha \in (0,1)$, corresponding to the Bernstein function  $f(\lambda)= \lambda ^\alpha$, and let $L^{\alpha}(t)$ denote its inverse; if $W(t)$ is a continuous time Markov chain on $\mathcal{S}^{d-1}$, whose waiting times are exponentially distributed with mean $1/\theta$, we construct the time-changed process
\begin{align}
V_\alpha(t):= W(L^\alpha (t)).\label{qqq}
\end{align}
By applying    (\ref{intertempi}), the  waiting times $J_n$ of $V_\alpha(t)$ follow the heavy tailed distribution
\begin{align}
P(J_n>t)=\mathds{E}e^ {-\theta L^\alpha (t)}= \mathcal{E}_\alpha  (-\theta t^{\alpha}) , \qquad \alpha  \in (0,1], \label{mittag}
\end{align}
where
$$  \mathcal{E}_\alpha (x):= \sum _{k=0}^\infty \frac{x^k}{\Gamma (1+\alpha k)} $$
denotes the so-called Mittag Leffler function. It is not hard to see that waiting times with these distribution have infinite expectation: one has indeed by \cite[Theorem 2.1]{meertoa} that
\begin{align}
P \l J_n > t \r  \, = \, \mathcal{E}_\alpha (-\theta t^\alpha) \, \stackrel{\infty}{\sim} \, \frac{1}{\theta} \frac{t^{-\alpha}}{\Gamma (1-\alpha)}
\label{511}
\end{align}
from which it easily follows that $\mathds{E}J_n=\infty$.

Note that the Markov case is formally re-obtained by putting $\alpha =1$, whence we obtain the exponential waiting times $P(J_n>t)=e^{-\theta t}$.

Denote now
\begin{align}
X_\alpha (t) := x+ \int_0^t V_\alpha (s) ds .
\end{align}
Note that in the one-dimensional case this process has already been considered in \cite{meoli}, where the authors studied several distributional properties (e.g., the distribution of the first passage time of the process).

We now prove that in this framework the transport process $(X_\alpha(t), V_\alpha(t))$ is governed by an equation which is formally similar to the Boltzmann equation \eqref{equazione coppia}, except that the material derivative $\partial_t -v\cdot \nabla _x$ on the left-hand side will be replaced by its fractional power $(\partial_t -v\cdot \nabla _x)^\alpha$.
This operator has been considered in several papers as a pseudo-differential operator with Laplace-Fourier $(x,t) \mapsto (\xi, \lambda)$ symbol $\l \lambda + i v \cdot \xi \r^\alpha$ and associated with the so-called L\'evy flights or L\'evy walks (e.g. \cite{magdalevy, fracmat}). The reader can consult, for example, \cite{barkai, klafter} for an introduction to the theory of L\'evy walks and some interesting applications.

We provide here the mathematical background of this idea by showing how these processes are related to semi-Markov random evolutions and how this fractional material derivative can be adjusted to be included in our general framework. In other words, we reformulate part of the theory of Section \ref{secre} in terms of pseudo-differential operators. 

So, let us consider the Banach space of functions $L^1([0,T] \times \mathbb{R}^d )$ endowed with its natural norm.
Moreover, let $H$ be the subset of $L^1([0,T] \times \mathbb{R}^d )$ whose elements have value zero at $t=0$ and are absolutely continuous functions such that their first derivatives (in time and space) are in $L^1([0,T] \times \mathbb{R}^d)$. Finally, let $T_v(t)$ be the translation operator, such that $T_v(s)h(x,t)=h(x+vs,t)$, and let $\Gamma _s$ denote the killed time shift, i.e., $\Gamma _s h(x,t)=h(x,t-s)\mathds {1}_{\left[s\leq t \right]}$. 
Consider the family of operators $\{ U_s^v\} _{s\geq 0}$
 defined by
\begin{align}
U_s^v h(x,t):=T_v(s)\Gamma _sh(x,t) = \begin{cases}
h(x+vs,t-s), \qquad & s\leq t, \\ 0,&s>t.
\end{cases}
\end{align}
By Lemma \ref{lemmadt}, $(U_s^v) _{s\geq 0}$ defines a strongly continuous contraction semigroup on $L^1([0,T] \times \mathbb{R}^d )$. On functions $h\in H$, the generator of $\{ U_s^v\} _{s\geq 0}$ has the form $-\frac{\partial}{\partial t}+v\cdot \nabla _x$. Consider the subordinate semigroup $(\mathfrak{U}_{s}^v )_{s \geq 0}$  defined by the following Bochner integral
$$ \mathfrak{U}_{s}^v h= \int _0^\infty U_{s'}^v   h \,  P(\sigma^\alpha(s)\in ds').
$$
By  Phillips theorem (see \cite{librobern},   Theorem 12.6), $(\mathfrak{U}_{s}^v) _{s \geq 0}$ is again a strongly continuous contraction semigroup on $L^1([0,T]\times \mathbb{R}^d)$ and, if $A^U_v$ is the generator of $U_s^v$, then the generator of $\mathfrak{U}_{s}^v$ is $-(-A^U_v)^\alpha$ which is defined at least on $Dom(A_v^U)$. For $h\in H$, since $A_v^U$ takes the form $-\frac{\partial}{\partial t}+v\cdot \nabla _x$, then  $-(-A^U_v)^\alpha$ has the form
\begin{align}
-\l\partial_t-v\cdot \nabla _x\r^\alpha\, h(x,t)  := &\int _0^\infty (U_s^vh(x,t)-h(x,t)) \frac{\alpha s^{-\alpha -1}}{\Gamma (1-\alpha)}ds \notag \\
= \, &  \int _0^\infty (h(x+vs, t-s)\mathds{1}_{[s \leq t]}-h(x,t)) \frac{\alpha s^{-\alpha -1}}{\Gamma (1-\alpha)}ds.
\label{bochpoint}
\end{align}
It turns out that we can take the last line in \eqref{bochpoint} as a Lebesgue integral, and use it to define the operator $\l \partial_t - v \cdot \nabla_x \r^\alpha$ as a Lebesgue integral.
We are now in position to prove the following result, which extends the linear Boltzmann equation \eqref{equazione coppia} to this kind of semi-Markov processes.

\begin{prop} \label{eq random flight}
Let $x \mapsto h(x,v)$ be in $L^1 \l \mathbb{R}^d \r$ for any $v$ and such that \revbt{$\sup_{v} \left\| h(\cdot,v) \right\|_{L^1}< +\infty$ and $\sup_{v}\left\| v \cdot \nabla_x h(\cdot,v) \right\|_{L^1}< +\infty$}. Denote
\begin{align}
q(x,v,t)= \mathds{E}^{v}h(X_\alpha(t), V_\alpha(t)). \label{expected}
\end{align}
Then $q(x,v,t)$ solves the following equation
\begin{align}
\l\partial_t-v\cdot \nabla _x\r^\alpha  q(x,v,t)- \frac{t^{-\alpha}}{\Gamma (1-\alpha)} h(x+vt,v) \, = \,  \theta \int _{S^{d-1}} (q(x,v',t)-q(x,v,t))\mu (dv')
\end{align}
where $\mu$ denotes the uniform measure on $\mathcal{S}^{d-1}$ and the operator $\l\partial_t-v\cdot \nabla _x\r^\alpha$ is defined as a Lebesgue integral, in the sense that the lhs exists for all $v \in S$ and almost all $(x,t)$.
Further, \revbt{it satisfies the equation in the Fourier-Laplace space}
\begin{align}
\l\lambda +i\xi \cdot v \r^\alpha \widehat{\widehat{q}}(\xi,v, \lambda)  \, = \, \l \lambda +  i\xi \cdot v \r^{\alpha-1}  \widehat{h}(\xi, v)+ \theta \int_{S^{d-1}} (\widehat{\widehat{q}}(\xi,v^\prime, \lambda)-\widehat{\widehat{q}}(\xi,v, \lambda))\mu (dv').
\label{eqpseudo}
\end{align}
\end{prop}
\begin{proof}
From the construction of the process and the law of total probability, by conditioning on the first jump time (see \cite[Theorem 3.1]{koro}) and using the Markov property of semi-Markov processes at jump times, the following renewal equation holds
\begin{align}
q(x,v,t) \, = \, h(x+vt,v)P^{v} \l J_1>t \r  + \int_0^t P^v \l J_1 \in ds \r \int_{S^{d-1}} q(x+vs,v^\prime, t-s) \mu(dv^\prime).
\label{5166}
\end{align}
\revbt{It follows by the assumptions on $h$ that, for any $T>0$,
\begin{align}
\int_0^T \int_{\mathbb{R}^d} |q(x,v,t)| dx \, dt \, \leq \, \sup_{v}\left\| h(\cdot,v) \right\|_{L^1 }  \, T \, < +\infty 
\end{align}
}and thus the Fourier-Laplace transform ($x \mapsto \xi, t \mapsto \lambda$) of $q(x, v, t)$ exists.
Hence we apply the Fourier transform in $x$ to both members of \eqref{5166}, to obtain
\begin{align*}
\widehat{q}(\xi,v,t)= e^{-i\xi \cdot  vt} \widehat{h}(\xi,v) P^v(J_1>t)+  \int _0^t P^v(J_1\in ds)e^{-i\xi \cdot vs}  \int _{S^{d-1}} \mu (dv')  \widehat{q}(\xi, v', t-s)  .
\end{align*}
This allows to have a convolution operator in $t$ on the right side. 
Then, by applying the Laplace transform in $t$, we get
\begin{align*}
\widehat{\widehat{q}}   (\xi,v,\lambda)= \widehat{h}(\xi,v)  \frac{(\lambda+i\xi \cdot v)^{\alpha -1}}{ \theta + (\lambda +i\xi \cdot  v)^{\alpha} }   +   \frac{\theta}{\theta + (\lambda +i\xi \cdot v)^\alpha}\int _{S^{d-1}} \mu (dv')  \widehat{\widehat{q}}(\xi, v', \lambda)  
\end{align*}
which can be re-arranged as
\begin{align*}
\widehat{\widehat{q}}   (\xi,v,\lambda) (\lambda +i\xi \cdot v)^{\alpha} =(\lambda +i\xi \cdot  v)^{\alpha -1}  \widehat{h}(\xi,v) +\theta\int _{S^{d-1}} (  \widehat{\widehat{q}}(\xi, v', \lambda)  - \widehat{\widehat{q}}(\xi, v, \lambda) )  \mu (dv').
\end{align*}
This proves that \eqref{eqpseudo} is verified.
%The fact that $\l \partial_t -v \cdot \nabla_x \r^\alpha q(x,v,t)$ is well defined comes from the fact that $(t,x) \mapsto q(x,v,t)$ is in $L^1 \l [0,T] \times \mathbb{R}^d \r$ since
%\begin{align}
%\int_0^T \int_{\mathbb{R}^d}  \left| \mathds{E}^vh(X_\alpha(t), V_\alpha(t)) \right| \, dx \, dt\,  \leq \, \int_0^T \int_{\mathbb{R}^d} \phi(x) dx \, dt = T C.
%\end{align}
Moreover, \revbt{by the assumption on $h$, we can apply the results of the previous section to say that} 
\begin{align}
(t,x) \mapsto \l \partial_t - v \cdot \nabla_x \r^\alpha q(x,v,t)
\end{align}
is in $L^1 \l [0,T] \times \mathbb{R}^d\r$ for any $T>0$ and $v \in \mathcal{S}^{d-1}$.
\revbt{We also note that}
\begin{align*}
&\int _0^\infty \int _{{\mathbb{R}}^d} e^{-\lambda t}e^{i\xi\cdot x} \l\partial_t-v\cdot \nabla _x\r^\alpha q(x,v,t) dxdt\\
& = \int _0^\infty \int _{\mathbb{R}^d} e^{-\lambda t}e^{i\xi\cdot x}   \biggl (   \int _0^\infty (-q(x+vs,v, t-s)\mathds{1}_{[s\leq t]}+q(x,v,t)) \frac{\alpha s^{-\alpha -1}}{\Gamma (1-\alpha)}ds        \biggr )           dxdt\\
&=\int _0^\infty e^{-\lambda t}  \biggl ( \int _0^\infty  (-e^{-i\xi\cdot v s} \widehat{q}(\xi,v, t-s)\mathds{1}_{[s\leq t]}+ \widehat {q}(\xi,v,t) )\frac{\alpha s^{-\alpha -1}}{\Gamma (1-\alpha)}ds      \biggr ) dt\\
& = \widehat{\widehat{q}}(\xi,v,\lambda) \int _0^\infty (1-e^{-s(\lambda +i\xi\cdot v)}) \frac{\alpha s^{-\alpha -1}}{\Gamma (1-\alpha)}ds\\
& = \widehat{\widehat{q}}(\xi,v, \lambda) (\lambda +i\xi\cdot v)^\alpha. 
\end{align*}

\end{proof}

\begin{os} \normalfont
We  observe that the distribution of $X_\alpha(t)$  has support on $||z-x||_d\leq t$ and  has a discrete component on the sphere $||z-x||_d=t$, the last one being the probability that $V_\alpha(t)$ has no jumps up to time $t$:
\begin{align*}
P(||X_\alpha(t)-x||_d=t)= \mathcal{E}_\alpha(-\theta t^\alpha)
\end{align*}
where $\mathcal{E}_\alpha(\cdot)$ is the Mittag-Leffler function.

Whenever \eqref{transition probability} has a Lebesgue density $p_t(z,w|x,v)$ on the open ball $||z-x||_d<t$ (this is true, for example, in the case of isotropic Markov transport process, consult \cite{stadje} for the explicit expression) by Proposition \ref{eq random flight} the following backward equation holds on the set $\left\| z-x\right\|_d\leq t$:
\begin{align}
&(\partial_t-v\cdot \nabla _x)^\alpha p_t(z,w|x,v)  -\delta(z-x-vt)\delta(w-v)\frac{t^{-\alpha}}{\Gamma (1-\alpha)}\notag \\
&= \theta \int _{S^{d-1}} ( p_t(z,w|x,v')-p_t(z,w|x,v) )\mu (dv'), 
\label{backtelfrac}
\end{align}
where $\delta(x)$ denotes the Dirac delta.
Hence, on the open set $||z-x||_d< t$, the density solves the following forward equation
\begin{align}
(\partial_t+w\cdot \nabla _z)^\alpha p_t(z,w|x,v)= \theta \int _{S^{d-1}} (p_t(z,w'|x,v)-p_t(z,w|x,v) )\mu (dw'). \label{forward}
\end{align}
\end{os}
\begin{os}
For another model of random flight process whose governing equations exhibit fractional operators, consult \cite{garra}.
\end{os}

\subsection{Convergence to a superdiffusive transport process}
\label{sectionconvergence}
We prove in this section that, under a suitable scaling, the process $X_\alpha(t)$, defined in the previous section, converges in distribution to a superdiffusive process. It turns out that this limiting process can be represented as a transport process with continuous paths.

In order to study a scaling limit  of $X_\alpha(t)$, we recall here some basic notions from the theory of CTRWs (see, for example,  \cite{meer annals probab, meerstra}). In this section we adopt the notation used for $X_\alpha(t)$, with the intepretation given in the CTRWs setting. 

Hence consider i.i.d. random vectors $(J_i, Y_i)$, where $Y_i \in \mathbb{R}^d$ represents a particle jump and $J_i \in \mathbb{R}^+$ is the waiting time preceeding that jump.
Let $Y(n)=Y_1+\dots +Y_n$ represent the location of the  particle after $n$ jumps and $\tau_n= J_1+\dots +J_n$    denote the time of the $n$-th jump. Moreover, let $N(t)=  \max \{n\geq 0: \tau_n\leq t \}$ be the renewal counting process representing the number of jumps up to time $t$. A CTRW is defined as
$$Y(N(t))=  \sum _{i=1}^{N(t)} Y_i.$$
 The wording ``coupled'' CTRWs means that $J_i$ and $Y_i$ are stochastically dependent.

By formula \eqref{processo posizione 2} we see that, apart from the initial position $x$, the process $X(t)$ can be written as a particular CTRW with waiting times $J_i$ and jumps $Y_i=J_iv_i$ (this is a L\'evy walk in the sense of \cite{magdalevy}) plus a corrective term
\begin{align*}
\epsilon(t) : = (t-\tau_{N(t)}  ) v_{N(t)+1}.
\end{align*}

Our goal is to understand wheter the scaled process
\begin{align}X_\alpha^c(t)& := x+c^{-H} \int_0^{ct} V_\alpha(s) ds     \label{scaled process}\notag \\
& = x+c ^{-H}   \sum _{i=1}^{N(ct)} J_i v_i + c^{-H}  \epsilon (ct)
\end{align}
converges, as $c\to \infty$, to some stochastic process for some $H>0$.

\begin{os} \normalfont \label{oss}
As mentioned in the previous section, if $J_i$ have finite expectation and variance, it is well known that a Brownian motion is obtained in the limit when $H=1/2$. The diffusive behavior is caused by some facts, which we recall here in a heuristic way, without claiming to be exhaustive. First note that 
when the r.v.'s $J_i$ have finite mean $1/\theta$ the renewal theorem states that $N(ct)\approx c\theta t$. Further $\left\| \epsilon(ct) \right\| \leq c^{-H}J_{N(ct)+1} \approx c^{-H} J_{c\theta t}$ and the latter quantity tends to zero a.s., for any $H>0$.
Now, if the $J_i$ also have finite
variance, putting $H=1/2$ and using central limit arguments, we have that  $c^{-1/2}Y([ct]) \to B(t)$, where $B(t)$ is a Brownian motion in $\mathbb{R}^d$.
Combining these two facts, we have   $$X^c(t)= x+c^{-1/2} Y(N(ct))+c^{-1/2} \epsilon (ct) \approx  x+ c^{-1/2} Y([c\theta t])   \to x+ B(\theta t),$$ hence the time-change simply yields a change of scale in the limit process.
\end{os}
In our case the r.v.'s $J_i$ are i.i.d. with infinite expectation. We show here that this permits to obtain a limit process with the scaling $x \to c^{-1/\alpha}x$, $t\to c^{1/\alpha}t$, $\alpha \in (0,1)$, i.e., putting $H=1$ in \eqref{scaled process} and renaming $c$ as $c^{1/\alpha}$ for convenience.

When studying this kind of limit of $X_\alpha(t)$, there are two problems to address. One is that the Mittag-Leffler distributed waiting times $J_i$ have infinite expectation and variance, so there is no hope to apply  the arguments  of Remark \ref{oss}. The second is that the quantity $c^{-1/\alpha}\epsilon (c^{1/\alpha}t)$ does not converge to zero and gives a contribution to the limit. 
In order to distinguish the two components we rewrite \eqref{scaled process} as
\begin{align}
X_\alpha^c (t) \,: = \, x+ Y_\alpha^c (t)+ \epsilon_c^\alpha (t)
\end{align}
where 
\begin{align}
Y_\alpha^c (t):= c^{-1/\alpha}Y(N(c^{1/\alpha}t)),
\label{5266}
\end{align}
and
\begin{align}
\epsilon_c^\alpha (t):= c^{-1/\alpha} \epsilon (c^{1/\alpha}t).
\end{align}
It is instructive to deal with the two components separately, therefore we first study $Y_\alpha^c$ and $\epsilon_c^\alpha$ and then we prove the convergence of the sum $X_\alpha^c$.

We begin with $Y_\alpha^c$ which is a particular CTRW whose waiting times belong to the domain of attraction of a stable law. In what follows, we refer to \cite{meer annals probab, meerstra} (and references therein) for the corresponding theory of scaling limits for this kind of coupled CTRWs.

Since $J_i$ has a Mittag-Leffler distribution, then the $J_iv_i$ belong to the domain of attraction of a rotationally invariant $\alpha$-stable law, hence 
$$ c^{-1/\alpha} \sum_{i=1}^{[ct]}J_iv_i \overset{\text{fdd}}{\to} A(t)$$
where $A(t)$ is a rotationally invariant stable process and $\overset{\text{fdd}}{\to}$ denotes convergence of all finite-dimensional distributions.
Moreover, the time $\tau_{n}$ of the $n$-th scattering and the renewal counting process $N(t)$ are such that 
$$
c^{-1/\alpha} \tau_{[ct]} \overset {\text{fdd}}{\to} \sigma^\alpha (t) \qquad c^{-1} N([c^{1/\alpha}t]) \overset{\text{fdd}}{\to} L^\alpha(t)
$$
where $\sigma^\alpha (t)$ is a $\alpha$-stable subordinator and $L^\alpha(t)$ is its inverse. The reader can consult \cite[Section 6.4]{FCbook}) for the above results. Heuristically, combining the above results, we should have for the scaled process
\begin{align} 
Y_\alpha^c(t)=  c^{-1/\alpha} \sum_{i=1}^{N(c^{1/\alpha}t)} J_i v_i  = \,  c^{-1/\alpha}\sum_{i=1}^{cc^{-1}N(c^{1/\alpha}t)} J_i v_i \approx A\l L^\alpha (t) \r.
\label{526}
\end{align}
Actually this is not exactly true: it will turn out that the process $A \l L^\alpha (t) \r$ is the limit of the overshooting CTRW, i.e., the process $$c^{-1/\alpha}\sum_{i=1}^{N(c^{1/\alpha}t)+1} J_iv_i$$
while our process converges to $A(L^\alpha(t)-)$. The  role of the following theorem is to make rigorous the above heuristic idea.
\begin{os} \normalfont
In particular, the theorem below will clarify what is the stochastic dependence between $L^\alpha(t)$ and $A(t)$, which can be described heuristically as follows. The process $L^\alpha(t)$ is the first passage time through level $t$ of a stable subordinator $\sigma^\alpha (t)$, while $A(t)$ is the sum of displacements whose direction is uniformly choosen on $\mathcal{S}^{d-1}$ and whose length is equal the jump of the stable subordinator $\sigma^\alpha(t)$. It follows that $\left\| A(L^\alpha(t)-) \right\|_d \leq \sigma^\alpha \l L^\alpha(t)- \r < t$, a.s., since $A(L^\alpha(t)-)$ is the position of the process $A(t)$ when the subordinator is in the starting point of the jump over $t$. It is clear from the construction of $A(t)$ that, on any finite interval of time, it performs a countable infinity of displacements with length less than $\delta$, for any $\delta$ and a finite number with length more than $\delta$, for some $\delta>0$.
\end{os}
Also, some physical properties of the limit process are interesting. It turns out that $A \l L^\alpha(t)- \r$ is superdiffusive (we recall that a process   is said to be superdiffusive if its mean square displacement grows as $t^{\gamma}$, with $\gamma >1$; in our case we find that $\gamma =2$).

\begin{te}
\label{teconvctrw}
Let $Y_\alpha^c(t)$ be the process defined in  \eqref{5266}. Suppose (without loss of generality), that the waiting time distribution \eqref{mittag} has $\theta =1$. Let $\mathpzc{e}_s$ be a Poisson point process on $\mathbb{R}^d \times (0, \infty)$ with the intensity 
\begin{align}
K(dx, ds) = \int_{S^{d-1}} \delta_{vs}(dx) \frac{\alpha s^{-\alpha-1}}{\Gamma (1-\alpha)}ds \,  \mu(dv).
\end{align}
Define the process on $\mathbb{R}^d \times (0, \infty)$
\begin{align}
\l A(t), \sigma^\alpha(t) \r := \sum_{s \leq t} \mathpzc{e}_s.
\end{align}
Then the following is true.

\begin{enumerate}

 \item $A(t)$ and $\sigma^\alpha(t)$ are stochastically dependent, with joint distribution given by
$$ \mathds{E}e^{-\lambda \sigma^\alpha(t)+i \xi\cdot A(t)} =e^{-t\psi (\xi,\lambda)} , \qquad \xi\in \mathbb{R}^d ,\lambda \geq 0,$$
where
\begin{align}
\psi (\xi,\lambda)=   \int _{S^{d-1}} (\lambda -i\xi \cdot v)^\alpha \mu (dv).
\end{align}
The marginal processes $A(t)$ and $\sigma^\alpha (t)$ are, respectively, a rotationally invariant $\alpha$-stable process and an $\alpha$-stable subordinator. 

 \item We have that $Y_\alpha^c(t)$ converges, in the sense of one-dimensional distributions, to the process $M(t):= A \l L^\alpha(t)- \r$ as $c\to \infty$ where $L^\alpha(t) := \inf \ll s \geq 0 : \sigma^\alpha(s) > t \rr$.

\item For all $t>0$, $M(t)$ has distribution with Fourier-Laplace transform
\begin{align}
\psi ^M(\xi , \lambda):=   \int _0^\infty e^{-\lambda t} \mathbb{e}\mathds{E}e^{i\xi\cdot M(t)}dt = \frac{\lambda ^{\alpha -1}}{\psi (\xi,\lambda)}. \label{psi M}
\end{align}

\item The limit process $M(t)$ is self-similar with scaling rate $M(ct)\overset{d}{=}cM(t)$, hence it has a superdiffusive behavior with mean square displacement
$$\mathds{E}\left\| M(t) \right\|^2_d=t^2\mathds{E}\left\| M(1) \right\|^2_d.$$

\end{enumerate}

\end{te}

\begin{proof}
We prove each item of the theorem separately.
\begin{enumerate}

\item     
Note that, by definition, we have
\begin{align}
\l A(t) , \sigma^\alpha (t) \r \, = \, \sum_{\mathcal{X} \in \Pi} g(\mathcal{X})
\end{align}
where $\Pi$ is a Poisson process on $\mathbb{R}^d \times (0, \infty) \times [0,t]$ with the mean measure $K(dx, ds) dt$ and $g(x,s,t)=(x,s)$.
Further
\begin{align}
\int_{\mathbb{R}^d \times (0, \infty) \times [0,t]} \l \left\| g(x,s,t) \right\|_{d+1} \wedge 1 \r  K(dx, ds) dt < \infty
\end{align}
and thus it follows from Campbell theorem (e.g., \cite[page 28]{king}) that, for $\underline{\xi}=(i\xi, -\lambda)$, $\xi \in \mathbb{R}^d$, $\lambda\geq 0$,
\begin{align}
\mathds{E}e^{\underline{\xi} \cdot \l A(t), \sigma^\alpha (t) \r} \, = \,& \mathds{E}e^{i\xi \cdot A(t) -\lambda \sigma^\alpha(t)} \notag \\
 = \,& e^{-t\int_{\mathbb{R}^d\times(0, \infty)} \l 1-e^{i\xi \cdot x-\lambda s} \r K(dx, ds)} \notag \\
= \, & e^{-t \int_{S^{d-1}}\l \lambda -i\xi \cdot v \r^\alpha \mu(dv)}.
\label{uscamp}
\end{align}
It follows from \eqref{uscamp}  that
\begin{align}
  &\mathds{E}e^{-\lambda \sigma^\alpha (t)} = e^{-t\lambda^\alpha} \\ &\mathds{E}e^{i\xi \cdot A(t)}=e^{-tB \left\| \xi \right\|_d^\alpha},
  \label{531}
\end{align} 
where $B$ is the constant depending on $\alpha$ and $d$, given by
\begin{align}
B:=\cos (\pi \alpha /2) \int_{S^{d-1}} | v \cdot \mathpzc{1}|^\alpha \mu(dv),
\end{align}
where $\mathpzc{1}$ is any unit vector.
In \eqref{531} we used \cite[Example 6.24]{FCbook} to say that
\begin{align}
\int_{S^{d-1}} \l - i \xi \cdot v \r^\alpha \mu(dv) \, = \, \int_{S^{d-1}} \l  i \xi \cdot v \r^\alpha \mu(dv) \, = \, B \left\| \xi \right\|_d^\alpha.
\end{align}
Therefore, marginally, $A(t)$ is a rotationally invariant $\alpha$-stable process while $\sigma(t)$ is an $\alpha$-stable subordinator.
\item Since $(J_i, v_i)$, $i=1\dots n$, are i.i.d., we can write the Fourier-Laplace transform of $n^{-1/\alpha} (Y(n), \tau_n   )$    in the following way
\begin{align} 
\mathds{E}\exp  \ll   n^{-1/\alpha} \bigl (-\lambda \tau_n+i\xi\cdot Y(n)   \bigr ) \rr                                                                                                                           =  \, &   \mathds{E}\exp \ll-\sum _{i=1}^n    n^{-1/\alpha}   \bigl (  \lambda -i\xi \cdot v _i    \bigr )J_i \rr   \notag \\
                                                                                                                          = \, & \biggl ( \mathds{E} \exp \ll-n^{-1/\alpha}  \bigl (\lambda -i\xi\cdot v_1 \bigr ) J_1 \rr    \biggr ) ^n. \label{f1}
\end{align}
By using conditional expectation, \eqref{f1} can be re-written as
\begin{align}                                                                                                                        &\l  \int _{S^{d-1}} \mathds{E} \left[\exp \{-n^{-1/\alpha}  \l\lambda -i\xi\cdot v_1 \r J_1 \} \mid \, v_1=v \right] \, \mu (dv) \r ^n \notag \\                                                                                                                         = \, & \l  \int _{S^{d-1}}     \frac{1}{1+ n^{-1}  (\lambda -i\xi\cdot v )^\alpha   }                         \mu (dv)       \r ^n \notag \\                                                                                                                          = \, &\l 1-\frac{1}{n} \int _{S^{d-1}} (\lambda -i\xi\cdot v)^\alpha \mu (dv) +o\l n^{-1} \r \r^n \notag \\                                                                                                                        \overset{n\to \infty}{\longrightarrow} \, &  \exp \biggl  \{ -  \int _{S^{d-1}} (\lambda -i\xi\cdot v)^\alpha \mu (dv)   \biggr \} \notag \\
                                                                                                                           = \, & e^{-\psi (\xi,\lambda)}                                                                                                                            \label{trovophi}                                                                                                                         
\end{align}
which means that $n^{-1/\alpha} \bigl ( Y(n) , \tau_n \bigr ) \overset{d}{\to} \bigl ( A(1), \sigma^\alpha (1)\bigr )$. 
Therefore  $c^{-1/\alpha} \bigl ( Y([ct], \tau_{[ct]}) \bigr ) \overset{d}{\to} \bigl ( A(t), \sigma^\alpha (t)\bigr )$, where the limit has Fourier-Laplace transform $\exp\ll -t\psi (\xi,\lambda)\rr$. The result then follows by using \cite[Theorem 3.4]{meer annals probab}: we have that $X_\alpha^c(t) \overset{d}{\to}A(L^\alpha(t)-)$ for any fixed $t>0$.

\item This comes from an application of \cite[Corollary 4.3]{meer annals probab}.

\item It is easy to check that the Fourier-Laplace transform of $M(t)$ satisfies the condition $c^{-1} \psi ^M(\xi,c^{-1}\lambda)=  \psi ^M (c\xi,\lambda)$, from which one has $M(ct)\overset{d}{=}cM(t)$ and therefore $M(t)\overset{d}{=} tM(1)$ so that $\mathds{E}\left\| M(t) \right\|^2_d= t^2\mathds{E}\left\|M(1)\right\|_d^2$. This concludes the proof.
\end{enumerate}
\end{proof}

\begin{os} \normalfont
Let $h(x,t)$ be a density of $M(t)$. Then, following the lines of  \cite[page 748]{meer annals probab}, Fourier-Laplace inversion of  (\ref{psi M}) gives, at least formally, the following equation
\begin{align}
\int _{S^{d-1}} \mu (dv) \l \partial_t -v\cdot \nabla _x \r ^\alpha  h(x,t)= h(x,0)\frac{t^{-\alpha}}{\Gamma (1-\alpha) }. \label{equazione limite}
\end{align}
To the best of our knowledge, \eqref{equazione limite} has never appeared. However, its one-dimensional version 
\begin{align} \label{equazione limite 1d}
\l \partial_t + \partial_x \r ^\alpha h(x,t) + \l \partial_t - \partial_x \r ^\alpha h(x,t) = 2h(x,0)\frac{t^{-\alpha}}{\Gamma (1-\alpha) } 
\end{align}
had already appeared in \cite{meer phys} in the study of one-dimensional continuous time random walks; our theory implies that \eqref{equazione limite 1d} governs the scaling limit of a one-dimensional transport process. Eq. \eqref{equazione limite 1d} was also studied in \cite{magdalevy} where the authors find explicit solutions.
\end{os}

We deal now with the component $\epsilon_c^\alpha$. Write $\epsilon (t)$ as
\begin{align}
\epsilon (t) = \l t-\tau_{N(t)} \r v_{N(t)+1} = \gamma (t) v_{N(t)+1}
\end{align}
where the process $\gamma (t)$ is the sojourn time in the current position of $V_\alpha(t)$. Since the r.v.'s $v_n$ are independent, $v_{N(t)+1}$ is a uniform r.v. on $\mathcal{S}^{d-1}$, independent from $\gamma (t)$, it follows that $\gamma (t) v_{N(t)+1}$ is a vector with norm $\gamma (t)$ and independent uniform orientation. In our process $\eqref{scaled process}$ it represents the last displacement of $X_\alpha(t)$ (after the last scattering). Therefore, in the limit, it converges to the last displacement of the process $A(t)$ with length $t-\sigma^\alpha(L^\alpha(t)-)$ and uniform orientation.

Consider the process 
\begin{align}
\epsilon_c^\alpha (t) = c^{-1/\alpha}\gamma (c^{1/\alpha} t) v_{N(c^{1/\alpha}t)+1}.
\end{align}
Since the orientation (velocity) is independent on $\gamma$ and its distribution is uniform independently on $c$ and $t$, when studying the convergence in distribution of $\epsilon^\alpha_c$ it is sufficient to study the norm $\gamma_c^\alpha (t):=c^{-1/\alpha}\gamma (c^{1/\alpha} t)$ separately. In the next result we show that $c^{-1/\alpha}\gamma(c^{1/\alpha }t)$, i.e., the length of the last scattering of the process $X_\alpha^c$, converges (in distribution) to the process 
\begin{align}
\gamma^\sigma(t):= t-\sigma^\alpha\l L^\alpha(t)- \r,
\end{align}
i.e., the current time passed from the last renewal of the process $A \l L^\alpha(t)- \r$ (see \cite[Section 2.1]{meerstra} for a thourough discussion on renewal times of semi-Markov processes.)
\begin{prop}
We have that $\gamma_c^\alpha (t)$ converges to $\gamma^\sigma(t)$ in distribution as $c \to \infty$ and therefore $\epsilon_c^\alpha (t)$ converges in distribution to $\gamma^\sigma(t) U$ where $U$ is an independent uniform random variable on $\mathcal{S}^{d-1}$.
\end{prop}
\begin{proof}
We prove the result under the assumption $\theta=1$, without loss of generality.
The proof can be conducted by computing directly the limit. The distribution of $\gamma (t)$ is given by (e.g., \cite[page 61]{cox})
\begin{align}
P \l \gamma (t) \in dw \r \, = \, \delta_{t}(dw) \mathcal{E}_\alpha(-t^\alpha) + \mathds{1}_{[w<t]} \mathpzc{h}(t-w) \, \mathcal{E}_{\alpha}(-w^\alpha) \, dw,
\end{align}
where $\mathpzc{h}(w)$ is the renewal density which can be computed explicitely by resorting to \eqref{potmeassub} and Theorem \ref{teconvctrw}. By combining the information we know indeed that
\begin{align}
\mathpzc{h}(dw) \, = \, \mathds{E}\int_0^\infty \mathds{1}_{[\sigma^\alpha(t) \in dw]}dt
\end{align}
where $\sigma^\alpha (t)$ is a stable subordinator. It is well known that for a stable subordinator one has that
\begin{align}
\mathds{E}\int_0^\infty \mathds{1}_{[\sigma^\alpha(t) \in dw]}dt \, = \, \frac{w^{\alpha-1}}{\Gamma(\alpha)
}dw.
\label{rendensex}
\end{align}
It follows that, for $z>0$,
\begin{align}
P \l c^{-1/\alpha} \gamma (c^{1/\alpha}t) >  z \r \, = \,& \mathds{1}_{[z \leq t]} \left[ \mathcal{E}_\alpha(-ct^\alpha)+   \int_{c^{1/\alpha}z}^{c^{1/\alpha}t}  \mathcal{E}_\alpha (-w^\alpha) \mathpzc{h}(c^{1/\alpha}t-w)dw \right].
\label{firspas}
\end{align}

Therefore we get from  \eqref{firspas}, using \eqref{rendensex} and after a change of variable,
\begin{align}
P \l c^{-1/\alpha} \gamma (c^{1/\alpha}t) >  z \r \, = \, & \mathds{1}_{[z \leq t]}  \left[  \mathcal{E}_\alpha(-ct^\alpha)+ c^{1/\alpha} \int_{z}^{t}  \mathcal{E}_\alpha (-cw^\alpha) \frac{\l c^{1/\alpha}(t-w) \r^{\alpha-1}}{\Gamma (\alpha)}dw \right] \notag \\
= \, &\mathds{1}_{[z\leq t]}\left[ \mathcal{E}_\alpha(-ct^\alpha)+  \frac{c}{\Gamma(\alpha)} \int_{z}^{t}  \mathcal{E}_\alpha (-cw^\alpha) (t-w)^{\alpha-1}dw \right].
\label{limitmarg}
\end{align}
Now, in order to compute the limit as $c \to \infty$, we use \eqref{511}. Indeed by repeatedly applying dominated convergence to the integral on the rhs of \eqref{limitmarg} (on the set $[z,t)$) we find
\begin{align}
\lim_{c\to \infty} P \l c^{-1/\alpha} \gamma (c^{1/\alpha}t) >  z\r \, = \, &\lim_{c \to \infty}  \mathds{1}_{[z \leq t]} \frac{c}{\Gamma (\alpha)} \int_z^t \frac{w^{-\alpha}}{c\Gamma(1-\alpha)} (t-w)^{\alpha-1}dw \notag \\
= \, & \mathds{1}_{[z\leq t]}\frac{1}{\Gamma (\alpha)\Gamma(1-\alpha)} \int_z^t  w^{-\alpha}(t-w)^{\alpha-1}dw.
\label{limitm}
\end{align}
By \cite[Lemma 1.10]{bertoins}, the distribution of $\sigma^\alpha \l L^\alpha (t)- \r$ is
\begin{align}
P \l \sigma^\alpha \l L^\alpha(t)- \r \in dw \r \, = \, \mathds{1}_{[0 < w < t]}\frac{1}{\Gamma  (\alpha) \Gamma (1-\alpha)} (t-w)^{-\alpha} w^{\alpha-1} dw,
\end{align}
whence $c^{-1/\alpha}\gamma(c^{1/\alpha}t) \stackrel{\text{d}}{\to}t-\sigma^\alpha\l L^\alpha (t)- \r$.
\end{proof}

So far, we dealt with the convergence in distribution of the two components $ Y_\alpha^c(t) \to A\l L^\alpha(t)- \r$ and $ \epsilon_c^\alpha(t) \to \gamma^\sigma (t)U$ separately. Now we show that their sum converges to the sum of the limit.
\begin{te}
We have that the process $X_\alpha^c(t)$ converges to $X_\infty (t):=A \l L^\alpha(t)- \r+\gamma^\sigma(t)U$, in the sense of one-dimensional distributions. Further the process $X_\infty (t)$ is superdiffusive with order $2$, i.e., $\mathds{E}\left\| X_\infty(t)\right\|_d^2 = Kt^2 $ for $K= \mathds {E} \left\|M(1) \right\|_d^2 +(1-\alpha)(2-\alpha)/2$.
\end{te}
\begin{proof}
The proof can be conducted by resorting again to \cite[Theorem 3.4]{meer annals probab}, as follows.
Consider the process $\l Y(N(t)), \, \gamma (t)\r$  in $\mathbb{R}^d\times (0,\infty)$, where $\gamma (t)=t-\tau_{N(t)}$. Note that
\begin{align*}
\l Y(N(t)), \, t-\tau _{N(t)} \r =(\bm{0},t)+ \sum _{i=1}^{N(t)} (Y_i, -J_i)
\end{align*}
has the form of a coupled continuous time random walk in $\mathbb{R}^d\times (-\infty, 0)$, plus a drift.
Take now the vector process $Z_n:= \l Y(n), -\tau_n \r$. In order to apply \cite[Theorem 3.4]{meer annals probab} we compute the Laplace-Fourier transform of $n^{-1/\alpha}\l Z(n), \tau_n \r$ by performing the same computation as in \eqref{trovophi}. We get indeed, for $\lambda \geq 0$, $\xi \in \mathbb{R}^{d}$, $k \in \mathbb{R}$,
\begin{align}
\mathds{E} e^{ n^{-1/\alpha}\l -\lambda \tau_n+i\xi \cdot Y(n)-ik\tau_n\r}  \, \to \, \exp \ll -\int_{S^{d-1}} \l \lambda +ik -i\xi \cdot v  \r^\alpha \mu(dv)\rr.
\end{align}
It follows that $n^{-1/\alpha}(Z_n, \tau_n) \to \l A(1), -\sigma^\alpha(1), \sigma^\alpha(1) \r$ and also
\begin{align}
c^{-1/\alpha}\l Z_{[ct]}, \tau_{[ct]} \r \stackrel{\text{d}}{\to} \l A(t), -\sigma^\alpha(t), \sigma^\alpha(t) \r.
\end{align}
Therefore, we have by \cite[Theorem 3.4]{meer annals probab} that 
\begin{align}
c^{-1/\alpha} \sum_{i=1}^{N(c^{1/\alpha}t)} \l Y_i, -J_i \r \stackrel{\text{d}}{\to}  \bigg( A\l L^\alpha(t)- \r,- \sigma^\alpha \l L^\alpha(t)- \r \bigg)
\end{align}
so that
\begin{align}
&c^{-1/\alpha} \l Y (N(c^{1/\alpha} t), c^{1/\alpha}t-\tau_{N(c^{1/\alpha}t)} \r \notag \\ = \,& (\bm{0},t)+ c^{-1/\alpha} \sum_{i=1}^{N(c^{1/\alpha}t)} \l Y_i, -J_i \r \notag \\ \stackrel{\text{d}}{\to} &  \, (\bm{0},t)+ \bigg( A\l L^\alpha(t)- \r,- \sigma^\alpha \l L^\alpha(t)- \r \bigg).
\end{align}
It follows that $\l Y_\alpha^c(t), \gamma_c^\alpha (t) \r \stackrel{\text{d}}{\to} \l A \l L^\alpha(t)- \r, \gamma^\sigma(t) \r$ and, equivalently, 
\begin{align}
\l Y_\alpha^c(t), \epsilon_c^\alpha (t) \r \stackrel{\text{d}}{\to}\l A \l L^\alpha(t)- \r, \gamma^\sigma(t)U \r.
\end{align}
By using  the continuous mapping theorem we have that the scaled transport process
\begin{align*}
X^c_\alpha (t) = c^{-1/\alpha} Y(N(c^{1/\alpha}t))+ c^{-1/\alpha} \gamma (c^{1/\alpha} t)U,
\end{align*}
is such that
\begin{align*}
X^c_\alpha (t) \overset{\text{d}}{\longrightarrow}               A(L^\alpha(t)-) + \bigl (t-\sigma^\alpha (L^\alpha(t)-)  \bigr)U
\end{align*}
for each $t\geq 0$.

The asymptotic behaviour of $X_\infty (t)$ can be obtained as follows: we have that
\begin{align}
\mathds{E} \left\| X_\infty (t) \right\|^2_d \, = \, \mathds{E} \left\| M(t)+\gamma^\sigma (t) U \right\|_d^2  \, = \,  &\mathds{E} \left\| M(t) \right\|_d^2 + \mathds{E} \left\| \gamma^\sigma (t) U \right\|_d^2
\label{finalmenteforse}
\end{align}
since $U$ is uniform, independent on $\gamma^\sigma$ and $M(t)$ and has zero expectation (with all its marginals). 
It is clear from \eqref{limitm} that $\gamma^\sigma(t) \stackrel{\text{d}}{=} t \beta$ where $\beta$ is a Beta r.v. with parameters $\alpha$ and $1-\alpha$. It follows from \eqref{finalmenteforse}, using Theorem \ref{teconvctrw}, that
\begin{align}
\mathds{E} \left\| M(t)+\gamma^\sigma (t) U \right\|_d^2  \, = \, t^2 \mathds{E} \left\| M(1) \right\|_d^2 + t^2 \mathds{E} \beta^2 = t^2 \l\mathds{E} \left\| M(1) \right\|_d^2 + (1-\alpha)(2-\alpha)/2\r.
\end{align}
\end{proof}
\begin{os} \normalfont
It is interesting to note that the proof of previous theorem could be conducted explicitely, computing directly the limit on the distribution of $\l Y_\alpha^c(t), \gamma_c^\alpha (t) \r$. However, since some computations  are cumbersome, we outline here only the main parts of it in order to specify the explicit distributions.
We have  that
\begin{align}
P \l Y(t) \in dx, \gamma(t) \in ds \r \, = \, \delta_t(ds) \mu_t(dx) \mathcal{E}_\alpha(-t^\alpha) + \mathds{1}_{[s<t]}\mathds{1}_{[\left\|x \right\|_d< t-s]} \mathpzc{h}(dx, t-ds) \mathcal{E}_\alpha(-s^\alpha),
\label{joint}
\end{align}
where $\mu_t(dx)$ is the uniform measure on $\mathcal{S}^{d-1}_t$ and
\begin{align}
\mathpzc{h}(dx, ds) \, : = \, \sum_{n=1}^\infty P\l \tau_n \in ds, Y(n) \in dx \r
\label{defrenjoint}
\end{align}
reprents the probability to have a jump of the process $Y(t)$ at time $t-s$ which ends at $x$. The last term in \eqref{joint} can be justified by observing that
\begin{align}
\mathpzc{h}(dx, t-ds) \mathcal{E}_\alpha(-s^\alpha) \, = \, & \sum_{n=1}^\infty P\l \tau_n \in t- ds, Y(n) \in dx \r P \l J_1 >s \r \notag \\
= \, & \sum_{n=1}^\infty P\l \tau_n \in t-ds, Y(n) \in dx \r  P \l \tau_{n+1}-\tau_n>s \r  \notag \\
= \, & P \l \bigcup_n \ll \tau_n \in t- ds, Y(n) \in dx, \tau_{n+1}-\tau_n>s \rr \r.
\label{555}
\end{align}
Use the notation $\mathpzc{h}(x,s)dxds=\mathpzc{h}(dx,ds)$ since it is clear from \eqref{555} that $\mathpzc{h}(dx,ds)$ is absolutely continuous on $\left\| x \right\|_d<s$.
It follows that, for $w \in \mathbb{R}^d$, $\epsilon >0$,
\begin{align}
&P \l Y_\alpha^c (t) > w, \gamma_c^\alpha (t) > \epsilon \r \notag \\
 = \,&\mathds{1}_{[\left\| w \right\|_d<t-\epsilon, \epsilon < t]}\left[ \mathcal{E}_\alpha \l -cs^\alpha \r +\int_{c^{1/\alpha}\epsilon}^{c^{1/\alpha}t} \int_{x>c^{1/\alpha}w, \left\| c^{1/\alpha} x\right\|_d < c^{1/\alpha} t-s} \mathpzc{h}(x, c^{1/\alpha}t-s) \mathcal{E}_\alpha (-s^\alpha) \, dx \,  ds \right] \notag \\
= \, &\mathds{1}_{[\left\| w \right\|_d<t-\epsilon, \epsilon < t]}\left[ \mathcal{E}_\alpha \l -cs^\alpha \r +c^{2/\alpha}\int_{\epsilon}^{t} \int_{x>w,\left\| x \right\|_d < t-s} \mathpzc{h}(c^{1/\alpha}x,c^{1/\alpha}(t-s)) \mathcal{E}_\alpha (-cs^\alpha)dx \, ds \right].
\label{joinbl}
\end{align}
By a simple argument, using \eqref{defrenjoint} and \eqref{trovophi}, we get that $\mathpzc{h}(c^{1/\alpha}x,c^{1/\alpha}t) \sim u(c^{1/\alpha}x,c^{1/\alpha}s)$, as $c \to \infty$, where $u^\alpha(x,s)$ is the potential density
\begin{align}
u^\alpha(x,s) dx ds :=\mathds{E}\int_0^\infty \mathds{1}_{[A(t) \in dx, \sigma^\alpha (t) \in ds]} dt.
\end{align}
Use this estimation together with \eqref{511}, applying dominated convergence properly to get that, as $c \to \infty$,
\begin{align}
&\lim_{c \to \infty}P \l Y_\alpha^c (t) > w, \gamma_c^\alpha (t) > \epsilon \r \notag \\ = \, &\mathds{1}_{[\left\| w \right\|_d<t-\epsilon, \epsilon < t]}\lim_{c \to \infty} c^{2/\alpha-1} \int_\epsilon^t \int_{x>w,\left\| x \right\|_d < t-s} u(c^{1/\alpha}x, c^{1/\alpha }(t-s)) \frac{s^{-\alpha}}{\Gamma (1-\alpha)} \, dx \, ds \notag \\
= \, & \mathds{1}_{[\left\| w \right\|_d<t-\epsilon, \epsilon < t]}\int_\epsilon^t \int_{x>w,\left\| x \right\|_d < t-s} u(x, t-s) \frac{s^{-\alpha}}{\Gamma (1-\alpha)}  \, dx \, ds
\label{jointlf}
\end{align}
where, in the last step we used self-similarity of $\l A(t), \sigma^\alpha (t) \r$, to say that $u(c^{1/\alpha}x, c^{1/\alpha} s) = c^{1-2/\alpha} u(x,s)$. The distribution in \eqref{jointlf} coincides with the distribution of $\l M(t), \gamma^\sigma (t) \r$ which was obtained in \cite[Remark 4.2]{meerstra}. To check this, observe that
\begin{align}
K \l  \mathbb{R}^d \times [s, \infty) \r \, = \, \frac{s^{-\alpha}}{\Gamma (1-\alpha)}.
\end{align}
\end{os}
\begin{os} \normalfont
It is interesting to note that the theory of CTRW limit \cite{meerstra} states that the
\begin{align}
c^{-1/\alpha}\sum_{i=1}^{N(c^{1/\alpha}t)} J_iv_i \to A \l L(t)- \r
\end{align}
while
\begin{align}
c^{-1/\alpha}\sum_{i=1}^{N(c^{1/\alpha}t)+1} J_iv_i \to A \l L(t) \r
\end{align}
in distribution. Further both processes have discontinuous paths. The process $X_\infty (t)$ instead has continuous paths: the particle is moving, for any $t$, from the point $A \l L(t) -\r$ to the point $A \l L(t) \r$ where a new displacement will be performed.
\end{os}

\subsection{On the telegraph process: an hyperbolic-type equation}
In the case $d=1$, the isotropic transport process $(X(t), V(t))$ takes values in $\mathbb{R}\times \{+1,-1\}$. In compact form, it can be defined by
\begin{align}
V(t)= V_0(-1)^{N(t)} \qquad X(t)= x+V_0\int _0^t (-1)^{N(\tau)}d\tau,
\end{align}
where $V_0$ takes values in $\{1,-1\}$ and $N(t)$ denotes the number of renewals up to time $t$.

The Markov case, i.e., the case where $N(t)$ is a Poisson process, is usually called telegraph process and has been widely studied in the literature (consult, e.g., \cite{goldstein, kac, ratanov} and \cite[Chapter 1]{Pinsky}); 
 it is useful to specify that such a  process can be constructed in two ways which are equivalent in terms of governing equations:

\textit{i)}   At random times governed by a Poisson process with intensity $2\theta$, the particle can either continue to move in the same direction or it can reverse direction with probability  $1/2$.

\textit{ii)} At random times governed by a Poisson process with intensity $\theta $, the particle reverses direction with probability 1.

These constructions are equivalent, in the sense that the semigroup of $(X(t), V(t))$ exhibits in both cases the same infinitesimal generator
\begin{align*}
G h(x, v)= v {\partial_ x} h(x, v)   + \theta \biggl  (h(x, -v) - h(x, v) \biggr ), \qquad   v=\pm 1.
\end{align*}

Moreover, concerning the Markov case, it is well known (see e.g. \cite{kac}) that the density of the continuous component of $X(t)$, say $p_t(z|x)= \partial_z P(X(t)<z|X(0)=x)$ for $|z-x|<t$, is the fundamental solution of the damped wave equation
\begin{align}
\partial _t ^2 p_t(z|x)=  \partial_ z^2 p_t(z|x)-2\theta \partial_ t p_t(z|x) \label{damped markov}
\end{align}
under the initial conditions $p_0(z|x)= \delta (z-x)$ and $\partial _t p_t(z|x)|_{t=0}=0$.
Due to homogeneity and isotropy, $p_t(z|x)$ depends on $z$ and $x$ through their difference, so (\ref{damped markov}) is equivalent to 
\begin{align}
\partial _t ^2 p_t(z|x)=  \partial_ x^2 p_t(z|x)-2\theta \partial_ t p_t(z|x) \label{damped markov 2}
\end{align}

We derive here, heuristically, a generalization of \eqref{damped markov 2} holding for the telegraph process with Mittag-Leffler waiting times (i.e., the one-dimensional version of $\l X_\alpha (t), V_\alpha (t) \r$ of previous section), where the related renewal counting process $N(t)$ is the so-called fractional Poisson process (consult, e.g., \cite{orsbegh, meerpoisson}). The forthcoming derivation is the heuristic version of the general case of section \ref{secdw}.
Indeed, consider the equation \eqref{backtelfrac} on the open set $|z-x|<t$ and sum both members in the variable $w \in \ll -1,1 \rr$; then  such equation splits into
\begin{align}
 \l \partial_t- \partial_x \r^\alpha p_t(z|x,1)&= \theta \bigl (p_t(z|x,-1)-p_t(z|x,1) \,\bigr ) \notag \\ 
\l \partial_t +\partial_x \r^\alpha p_t(z|x,-1)&= \theta \bigl (p_t(z|x,1)-p_t(z|x,-1) \bigr ) \label{due}
\end{align}
where $p_t(z|x,v):=\partial_z P(X(t)<z|X(0)=x, V(0)=v)$, with $v=\pm1$.
Now, let 
\begin{align*}
\l \partial_t^2 -  \partial_x^2 \r ^\alpha := \l \partial_t -\partial_x \r^\alpha      \l\partial_t+\partial_x \r^\alpha
\end{align*}
be the fractional version of the D'Alembert operator (see \cite[pages 554 -- 555]{samko}), having Fourier-Laplace symbol
\begin{align}
(\lambda+i\xi)^\alpha (\lambda -i\xi)^\alpha = (\lambda ^2+\xi^2)^\alpha.
\end{align}
Using that $p_t(z|x)= \frac{1}{2} \bigl (p_t(z|x,-1)+p_t(z|x,1) \bigr )$, by  simple algebraic manipulations,  equations \eqref{due} can be re-arranged as
\begin{align}
\l \partial_t^2-   \partial_x^2 \r^\alpha p_t(z|x) = -\theta \l \partial_t -\partial_x \r^\alpha p_t(z|x)- \theta \l \partial_t + \partial_x \r^\alpha p_t(z|x) \label{equazione dalambertiano},
\end{align}
which is the fractional version of   (\ref{damped markov 2})   and          formally reduces to   (\ref{damped markov 2})     when $\alpha =1$.
\begin{os} \normalfont
In the paper \cite{orsbegh} the authors studied another random flight driven by a fractional Poisson process, hence having Mittag-Leffler waiting times. However, such a process is obtained by the time-change of the position process and thus it strongly differs, e.g., pathwise, from our process $X_\alpha (t)$. \revbt{The reader should compare the results in this section with \cite{fedotov} where the author derives the governing equation for the probability density function of a classical L\'evy walk. It turns out that this equation involves a classical wave
operator together with memory integrals (induced by the spatiotemporal coupling) and therefore it can be viewed as an alternative to \eqref{equazione dalambertiano}}.
\end{os}

\section*{Acknowledgements}
The authors are grateful to an anonymous referee, whose remarks and suggestions have improved a previous draft of the paper.

The author Bruno Toaldo is partially supported by Gruppo Nazionale per l’Analisi Matematica, la Probabilit\`{a} e le loro Applicazioni (GNAMPA-INdAM).

\vspace{1cm}
\end{document}